\documentclass[reqno]{amsart}

\usepackage[hyphens]{url}
\usepackage{graphicx,mathtools,enumitem,xcolor}
\usepackage{amssymb}

\theoremstyle{plain}
\newtheorem{lemma}{Lemma}
\newtheorem{theorem}[lemma]{Theorem}
\newtheorem{corollary}[lemma]{Corollary}
\newtheorem{definition}[lemma]{Definition}
\newtheorem{step}{Step}

\theoremstyle{remark}
\newtheorem{remark}{Remark}

\newtheorem{case}{Case}

\numberwithin{equation}{section}
\newcommand{\noqed}{\renewcommand{\qed}{}}

\newcommand {\eps}{\varepsilon}

\newcommand  {\R}{\mathbb{R}}
\newcommand  {\I}{\mathbb{I}}

\renewcommand{\d}{\mathrm{d}}
\newcommand  {\e}{\mathrm{e}}

\newcommand  {\Fc}{\mathcal{F}}

\newcommand  {\Ic}{\mathcal{I}}

\newcommand  {\Cc}{\mathcal{C}}

\newcommand  {\Pc}{\mathcal{P}}

\newcommand  {\Chi}{\mathcal{X}}

\newcommand{\erfc}{\operatorname{erfc}}

\newcommand{\unique}{{\operatorname{unique}}}
\newcommand{\normal}{{\operatorname{normal}}}
\newcommand{\jump}{{\operatorname{jump}}}
\renewcommand{\deg}{{\operatorname{deg}}}
\newcommand{\reg}{{\operatorname{reg}}}

\newcommand{\RD}{{\operatorname{RD}}}

\newcommand{\ES}{{\operatorname{ES}}}

\newcommand{\ringwidth}{L}

\DeclareMathOperator*{\esssup}{ess\,sup}

\newcommand*{\abs}[1]{\lvert #1 \rvert}

\begin{document}
\title[Conditional uniqueness of solutions to the Keller--Rubinow model]%
{Conditional uniqueness of solutions to the Keller--Rubinow model
  for Liesegang rings \\ in the fast reaction limit}

\author[Z. Darbenas]{Zymantas Darbenas}
\email[Z. Darbenas]{z.darbenas@jacobs-university.de}
\author[R. v.~d.~Hout]{Rein van der Hout}
\email[R. v.~d.~Hout]{rein.vanderhout@gmail.com}
\author[M. Oliver]{Marcel Oliver}
\email[M. Oliver]{m.oliver@jacobs-university.de}

\address[Z. Darbenas and M. Oliver]%
{School of Engineering and Science \\
 Jacobs University \\
 28759 Bremen \\
 Germany}

\address[R. v.~d.~Hout]%
{Dunolaan 39 \\
6869VB Heveadorp \\
The Netherlands}

\date{\today}

\begin{abstract}
We study the question of uniqueness of weak solution to the fast
reaction limit of the Keller and Rubinow model for Liesegang rings as
introduced by Hilhorst \emph{et al.}\ (J. Stat.\ Phys.\ 135, 2009,
pp.\ 107--132).  The model is characterized by a discontinuous
reaction term which can be seen as an instance of spatially
distributed non-ideal relay hysteresis.  In general, uniqueness of
solutions for such models is conditional on certain transversality
conditions.  For the model studied here, we give an explicit
description of the precipitation boundary which gives rise to two
scenarios for non-uniqueness, which we term ``spontaneous
precipitation'' and ``entanglement''.  Spontaneous precipitation can
be easily dismissed by an additional, physically reasonable criterion
in the concept of weak solution.  The second scenario is one where the
precipitation boundaries of two distinct solutions cannot be ordered
in any neighborhood of some point on their common precipitation
boundary.  We show that for a finite, possibly short interval of time,
solutions are unique.  Beyond this point, unique continuation is
subject to a spatial or temporal transversality condition.  The
temporal transversality condition takes the same form that would be
expected for a simple multicomponent semilinear ODE with discontinuous
reaction terms.
\end{abstract}

\maketitle


\section{Introduction}

We study the question of uniqueness of weak solution to the fast
reaction limit of the Keller and Rubinow model for Liesegang rings,
\begin{subequations}
  \label{e.original}
\begin{gather}
  u_t = u_{xx} +
        \frac{\alpha \beta}{2 \sqrt t} \, \delta (x - \alpha \sqrt{t})
        - p[x,t;u] \, u \,,
  \label{e.original.a} \\
  u_x(0,t) = 0 \quad \text{for } t \geq 0 \,, \label{e.original.b} \\
  u(x,0) = 0 \quad \text{for } x>0 \label{e.original.c}
\end{gather}
where the precipitation function $p[x,t;u]$ depends on $x$, $t$, and
nonlocally on $u$ via
\begin{equation}
  p[x,t;u] = H
  \biggl(
    \int_0^t (u(x,\tau) - u^*)_+ \, \d \tau
  \biggr) \,.
  \label{e.hhmo-p}
\end{equation}
\end{subequations}
Here, $H$ denotes the Heaviside function with the convention that
$H(0)=0$ and $u^*$ denotes the supersaturation concentration.

The model was derived by Hilhorst \emph{et al.}\
\cite{HilhorstHM:2007:FastRL,HilhorstHM:2009:MathematicalSO}, based on
earlier work in
\cite{HilhorstHP:1996:FastRL,HilhorstHP:1997:DiffusionPF}, from a
three-component two-stage system of reaction-diffusion equations due
to Keller and Rubinow \cite{KellerR:1981:RecurrentPL} under the
assumption that one of the first-stage reactants does not diffuse,
that the lower threshold of criticality is zero, and that the reaction
constant of the first-stage reaction is large.  In the following, we
shall refer to the reduced model \eqref{e.original} as the HHMO-model.

Hilhorst \emph{et al.}\
\cite{HilhorstHM:2007:FastRL,HilhorstHM:2009:MathematicalSO}
introduced and proved existence of weak solutions to
\eqref{e.original}.  Modulo technical details, weak solutions are
pairs $(u,p)$ that satisfy \eqref{e.original.a} integrated against a
suitable test function such that
\begin{equation}
\label{e.hhmo-p-weak}
  p(x,t) \in H
  \biggl(
    \int_0^t (u(x,\tau) - u^*)_+ \, \d \tau
  \biggr)
\end{equation}
where $H$ denotes the Heaviside \emph{graph}
\begin{equation}
\label{p.H.def}
  H(y) \in
  \begin{cases}
    0 & \text{when } y<0 \,, \\
    [0,1] & \text{when }y=0 \,, \\
    1 & \text{when } y>0 \,,
  \end{cases}
\end{equation}
subject to the additional requirement that $p(x,t)$ takes the value
$0$ whenever $u(x,s)$ is strictly less than the threshold $u^*$ for
all $s\in[0,t]$.  This constraint can be stated as
\begin{equation}
\label{e.hhmo-p-weak-alternative}
  p(x,t)\in
  \begin{cases}
     0&\text{ if }\sup_{s\in[0,t]}u(x,s)<u^* \,,\\
     [0,1]&\text{ if }\sup_{s\in[0,t]}u(x,s)=u^* \,,\\
     1 &\text{ if }\sup_{s\in[0,t]}u(x,s)>u^* \,.
  \end{cases}
\end{equation}

The problem left open by \cite{HilhorstHM:2009:MathematicalSO} is the
question of uniqueness of weak solutions to the HHMO-model.  The main
obstacle is that the precipitation term is neither Lipschitz
continuous nor local in time.  Moreover, it may not even be monotonic
in the following sense.  If $u_1$ and $u_2$ are weak solutions with
associated precipitation functions $p_1$ and $p_2$, it is not clear
whether
\begin{equation}
  (p_1 \, u_1 - p_2 \, u_2) \, (u_1-u_2) \ge 0
  \label{e.monotonicity}
\end{equation}
a.e.\ in space-time.  An estimate of this form would imply uniqueness
by standard energy methods.  We remark that for other models involving
phase transitions, e.g.\ for moist advection in models of the
atmosphere with humidity and saturation
\cite{CotiZelatiFT:2013:EquationsAH,TemamT:2016:EquationsMA},
monotonicity can be asserted.  The behavior of the precipitation
function is an instance of a one-sided non-ideal relay.  In general,
non-ideal relays switch from an ``off-state'' $0$ to the ``on-state''
$1$ when the input crosses a threshold $\mu$, and switches back to
zero only when the input drops below a lower threshold $\lambda<\mu$.
Here, the lower threshold is $\lambda=0$, so the relay never switches
back.  There are different ways of defining the behavior of non-ideal
relays; see, e.g., the brief survey in \cite{CurranGT:2016:RecentAR}.
The formulations differ in their behavior when the input reaches, but
does not exceed the relay threshold.  The three options described in
\cite{CurranGT:2016:RecentAR} are: (i) The relay switches as soon as
the threshold is reached
\cite{GurevichR:2013:WellPE,GurevichST:2013:ReactionDE,KrasnoselskiiP:1989:SystemsH,Visintin:1994:DifferentialMH},
(ii) the relay switches only when the threshold is exceeded,
attributed to Alt \cite{Alt:1985:ThermostatP}, or (iii) may take
intermediate values at the threshold subject to certain monotonicity
constraints, which are referred to as a \emph{completed relay}
\cite{AikiK:2008:MathematicalMB,Visintin:1986:EvolutionPH}.  All these
formulations are ``rate independent'', i.e., the state of the relay
only depends on the past and present values of the input, but not on
their rate of change.  All rate-independent formulations have issues
regarding their well-posedness in cases of non-transversal crossings
of the threshold.

The uniqueness issue can be illustrated with a simple system of two
ordinary differential equations, but extends to the case of spatially
distributed relays, including the HHMO-model as a reaction-diffusion
equation with precipitation.  For simplicity, we translate the
crossing of the critical threshold into the origin and look at the
non-autonomous system
\begin{subequations}
\begin{gather}
  \dot u(t) = f(t) + u(t) + v(t) - p_u(t) \,, \\
  \dot v(t) = f(t) + u(t) + v(t) - p_v(t) \,, \\
  u(0) = v(0) = 0 \,.
\end{gather}
\end{subequations}
Here, $p_u$ and $p_v$ denote the precipitation condition
\eqref{e.hhmo-p-weak-alternative} with $u^*=0$ for $u$ and $v$,
respectively.  If $p_u$ and $p_v$ are permitted to assume fractional
values, there is no hope for uniqueness, so the question here is
whether the restriction of $p_u$ and $p_v$ to binary values suffices
to select a unique solution.

Let us first consider the case $f(t)=\tfrac12$.  In this case, the
vector field without the precipitation terms is positive in both
components at time $t=0$ when the threshold is touched; we speak of a
\emph{transversal} crossing.  We see that both precipitation functions
must switch from zero to one at that instant.  Indeed, if none of the
precipitation functions switches, the solution is
$u(t)=v(t) = (\exp(2t)-1)/4 >0$ on some interval of positive time,
which violates \eqref{e.hhmo-p-weak-alternative}.  If one of the
precipitation function switches, $p_u$ say, the solution is
$u(t)=-t/2$, $v(t)=t/2$, so that the precipitation condition is still
violated on some interval of positive time.  So there is no choice and
both must switch.

If, on the other hand, $f(t)=t$, the vector field without the
precipitation terms is zero in both components at time $t=0$ when the
threshold is touched; we speak of a \emph{non-transversal} crossing.
Again, it is easy to see that at least one of the precipitation
functions must switch at $t=0$ for if not, both $u$ and $v$ will be
positive for $t>0$, violating the precipitation condition
\eqref{e.hhmo-p-weak-alternative}.  However, suppose that $p_u$
switches to $1$ at $t=0$, while $p_v$ remains zero.  Then $u(t)=-t$
and $v(t)=0$, which is a feasible solution.  Due to the symmetry,
$p_u=0$ and $p_v=1$ also gives a feasible solution.

We remark that as soon as fractional values are permitted, there are
further feasible solutions: in the transversal example, e.g.\
$p_u=p_v=\tfrac12$ is feasible, in the non-transversal example, any
convex combination $p_u=\lambda$ and $p_v=1-\lambda$ with
$\lambda \in [0,1]$ gives a feasible solution.  We believe that a
better disambiguation criterion would permit fractional values of the
precipitation function augmented by a suitable minimality condition.
This, however, is not trivial and outside of the scope of this paper.
For the present paper on the HHMO-model, we avoid this discussion
altogether by proving that, on some positive interval of time, the
precipitation function of a weak solution is essentially binary, i.e.\
binary except perhaps for values on a space-time set of measure zero.

Our results are the following.  We identify two scenarios for
non-uniqueness, ``spontaneous precipitation'' and ``entanglement''.
Spontaneous precipitation can be easily dismissed by an additional,
physically reasonable criterion in the concept of weak solution.
Entanglement is a scenario where there exists a point on the common
precipitation boundary such that in every neighborhood of this point
there are subregions where each one of two non-unique boundary curves
is ahead of the other.  To dismiss the second scenario, we perform a
detailed study of the topological and analytic properties of the
precipitation boundary.  Our results are two-fold.  First, there
exists an initial interval of time where monotonicity in the sense of
\eqref{e.monotonicity}, hence uniqueness, holds true.  Second, we
state a transversality condition, namely that the temporal rate of
change of concentration is non-degenerate at the precipitation
boundary, which prevents entanglement and implies monotonicity, hence
uniqueness.  Our analysis is restricted to a region where the solution
consists of a succession of distinct precipitation rings, the
\emph{ring domain}.  In numerical simulations of a range of models,
including the HHMO-model and the full Keller--Rubinow model, the ring
domain appears to persist for only a finite interval of time, longer
than our initial interval of uniqueness; breakdown of the ring domain
is proved for a simplified version of the HHMO-model in
\cite{DarbenasO:2021:BreakdownLP}.  After that, solutions may become
topologically even more complex and our methods do not apply.  For the
simplified model in \cite{DarbenasO:2021:BreakdownLP}, a reduction of
the problem to a scalar integral equation is possible and the question
of uniqueness can be answered in the affirmative in a class of
solutions that excludes accumulation of precipitation rings in reverse
time \cite{DarbenasO:2019:UniquenessSW}.  For the HHMO-model itself,
this reduction is not possible and the question remains open.

The paper is structured as follows.  In Section~\ref{s.weak}, we
review the concept of weak solutions, their basic properties, and show
that there are weak solutions whose precipitation function is not
changing at a point after the reactant source has passed.  In
Section~\ref{s.local-properties}, we introduce the ``ring domain'', a
non-empty region in which the solution can be characterized by
distinct precipitation bands, and prove a number of topological and
analytic properties of the precipitation boundary on the ring
domain.  In particular, we show that the precipitation function can be
given a canonical form up to changes on space-time sets of measure
zero.  In Section~\ref{sec:u.dif.cont} we present a boot-strap
argument that guarantees existence and continuity of a classical time
derivative away from the precipitation boundary and give a sufficient
condition that ensures existence and continuity of the time derivative
on the precipitation boundary as well.  Finally, uniqueness is proved
in Section~\ref{the.uniqueness}, unconditionally up to a finite,
possibly small time and under a temporal transversality condition on
the entire ring domain.

\section{Weak solutions}
\label{s.weak}

To begin, we note that without the precipitation term,
\eqref{e.original} has the explicit solution
\begin{gather}
\label{psi.def}
  \psi(x,t) = \Psi \Bigl( \frac{x}{\sqrt t} \Bigr)
\end{gather}
where
\begin{align}
  \label{self.similar.p0}
  \Psi(\eta)
  & = \frac{\alpha \beta \sqrt\pi}2 \, \e^{\tfrac{\alpha^2}4} \cdot
      \begin{dcases}
        \erfc(\alpha/2) & \text{if } \eta \leq \alpha \,, \\
        \erfc(\eta/2) & \text{if } \eta > \alpha \,.
      \end{dcases}
\end{align}
For further reference, we also define the standard heat kernel
\begin{equation}
  \Phi(x,t) = \begin{dcases}\frac1{\sqrt{4\pi t}} \,
                 \e^{-\tfrac{x^2}{4t}}&\text{if }t>0 \,,\\
		 0&\text{if }t\le0 \,.
	      \end{dcases}
\end{equation}

In the following definition of weak solution, we follow
\cite{HilhorstHM:2009:MathematicalSO,DarbenasHO:2018:LongTA} and
extend the spatial domain to the entire real line by even reflection.
In the main body of the paper, however, it is easier to formulate all
arguments and definitions exclusively on the first quadrant of the
$x$-$t$ plane.  Due to the implied even symmetry, we may still refer
to the fields at $x<0$ when convenient, in particular when stating
arguments based on the Duhamel principle.

\begin{definition}
\label{weak.sol.def}
A \emph{weak solution} to problem \eqref{e.original} is a pair $(u,p)$
satisfying
\begin{enumerate}[label={\upshape(\roman*)}]
  \item $u$ and $p$ are even in $x$, i.e.\ $u(x,t)=u(-x,t)$ and
  $p(x,t)=p(-x,t)$ for all $x \in \R$ and $t\ge0$,
  \item\label{weak.ii} $u-\psi\in C^{1,0}(\R\times[0,T])\cap L^{\infty}(\R\times[0,T])$ for
  every $T>0$,
  \item\label{weak.iii} $p$ is measurable, defined pointwise, and
  satisfies \eqref{e.hhmo-p-weak-alternative},
\item\label{weak.3.5} $p(x,t)$ is non-decreasing in time $t$ for every
$x \in \R$,
\item\label{weak.iv} the relation
\begin{equation}
\label{weak.sol.def.eq}
  \int_0^T\int_\R\varphi_t \, (u-\psi) \, \d y \, \d s
  = \int_0^T \int_\R
    \bigl(
      \varphi_x \, (u-\psi)_x + p \, u \, \varphi
    \bigr) \, \d y \, \d s
\end{equation}
holds for every $\varphi\in C^{1,1}(\R\times[0,T])$ that vanishes for
large values of $|x|$ and for time $t=T$.
\end{enumerate}
\end{definition}

The following additional notation is used throughout the paper.  We
define
\begin{subequations}
\begin{equation}
  \Pc=\{(x,t) \colon \alpha^2 \, t=x^2\}
\end{equation}
to denote the parabola on which the point source moves, and write 
\begin{align}
  D_o&=\{(x,t) \colon 0< x^2 < \alpha^2 \, t \} \,, \\
  D_u&=\{(x,t) \colon 0 < \alpha^2 \, t<x^2 \} 
\end{align}
\end{subequations}
to denote the open region of the upper half-plane over and under the
parabola $\Pc$, respectively.  Moreover, we formalize the notion of
precipitation ring and interring (or gap) as follows.

\begin{definition}
\label{ring}
The interval $[a,b]$ with $b>a>0$ is a \emph{ring} if the set
\begin{equation}
  \{ (y,s) \colon y \in[a-\eps_1,b+\eps_2], 
     \alpha^2 \, s\ge y^2,\, p(y,s)<1 \}
  \subset \R^2
\end{equation}
with non-negative $\eps_1$, $\eps_2$ has measure zero if and
only if $\eps_1=\eps_2=0$.

When $b>a=0$, the interval $[0,b]$ is a \emph{ring} if the set 
\begin{equation}
  \{(y,s) \colon y\in[0,b+\eps_1], \alpha^2 \, s\ge y^2,\, p(y,s)<1 \}
  \subset \R^2
\end{equation}
with non-negative $\eps_1$ has measure zero if and only if
$\eps_1=0$.
\end{definition}

\begin{remark}
If $[a,b]$ is a ring and $x \in [a,b]$---we say that ``$x$ is
contained in a ring''---then
$\max_{s\in[0, x^2/\alpha^2]} u(x,s) \ge u^*$.  For if not, by
continuity of $u$, there would be a neighborhood of the line segment
$\{x\} \times [0, x^2/\alpha^2]$ on which $u<u^*$.  But for a weak
solution satisfying property (P), this means that $x$ cannot be
contained in a ring.  In Section~\ref{s.local-properties} we shall
show that $\max_{s\in[0, x^2/\alpha^2]} u(x,s) = u^*$ only if the
maximum is taken exclusively on $\Pc$, we then speak of a
\emph{degenerate} precipitation boundary point.
\end{remark}

\begin{definition}
\label{interring}
The interval $[a,b]$ with $b>a>0$ is an \emph{interring} if the set
\begin{equation}
  \{(y,s) \colon y\in[a-\eps_1,b+\eps_2],\, p(y,s)>0 \}
  \subset \R^2
\end{equation}
with non-negative $\eps_1$, $\eps_2$ has measure zero if and
only if $\eps_1=\eps_2=0$.

When $0<a$, the interval $[a,\infty)$ is an \emph{interring} if the
set
\begin{equation}
  \{(y,s) \colon y\ge a-\eps, p(y,s)>0 \}
  \subset \R^2
\end{equation}
with non-negative $\eps$ has measure zero if and only if
$\eps=0$.
\end{definition}

When $u^*\ge\Psi(\alpha)$, the so-called subcritical or marginal
cases, it is not possible to have a recurrent pattern of rings and
interrings.  In these cases, weak solutions have a simple structure
which is completely described by \cite[Theorems~4
and~5]{DarbenasHO:2018:LongTA}.  Therefore we focus on the interesting
\emph{supercritical} case where $u^*<\Psi(\alpha)$.  In this case,
the following lemma asserts that at least a first precipitation ring
always exists.

\begin{lemma}[Existence of a first precipitation ring]
\label{one.ring}
Every weak solution to equation \eqref{e.original} with supercritical
precipitation threshold $u^*$ has at least one precipitation ring of
width at least $X_1 \ge \ringwidth$, where
\begin{equation}
  \ringwidth=\sqrt{\frac{\Psi(\alpha)-u^*}{\Psi(\alpha)}} \,.
\end{equation}
In particular, there is no interring of the form $[0,d]$.
\end{lemma}

\begin{proof}
First, recall \cite[Lemma~3.5]{HilhorstHM:2009:MathematicalSO}, which
states that
\begin{equation}
  u(x,t)\ge\psi(x,t)-\Psi(\alpha) \, t \,.
  \label{e.hhmo-lem3.5}
\end{equation}
Second, note that there exists a $t^*>0$ such that
\begin{equation}
  u^*=\Psi(\alpha)-\Psi(\alpha) \, t^* \,.
\end{equation}
Hence, \eqref{e.hhmo-lem3.5} implies that if $(x,t)\in\Pc$ with
$t<t^*$, then
\begin{equation}
  u(x,t)\ge\psi(x,t)-\Psi(\alpha)t
  >\Psi(\alpha)-\Psi(\alpha)t^*=u^* \,.
  \label{e.ulowerbound}
\end{equation}
In other words, $u$ is strictly greater than the precipitation
threshold $u^*$ on all points of the parabola $\Pc$ with $t<t^*$.
Now let 
\begin{equation}
  X_1 = \sup \bigl\{ x \colon 
    m\{(y,s) \colon y\in[0,x], \alpha^2 \, s\ge y^2, p(y,s)<1 \}
    = 0 \bigr\} \,.
\end{equation}
Then, $[0,X_1]$ is a ring according to Definition~\ref{ring} of width
$X_1 \geq \alpha\sqrt{t^*} \equiv \ringwidth$. 
\end{proof}

When the concentration reaches, but does not exceed the precipitation
threshold on sets of positive measure, which, as we shall show in
Section~\ref{s.local-properties}, is restricted to the region $D_o$,
``spontaneous precipitation'' might occur: at some time horizon $t$,
the precipitation function switches on a subset of
$\{x \colon u(x,t) = u^*\}$ of positive measure from $0$ to $1$.  In
\cite[Remark~3]{DarbenasHO:2018:LongTA}, we demonstrate that, at least
for the case of a marginal precipitation threshold, this possibility
is real.  To exclude non-uniqueness by spontaneous precipitation, we
pose the following additional restriction on weak solutions:
\begin{itemize}
\item [(P)] There exists a measurable function $p^*$ such that for
a.e.\ $x \in \R_+$,
\begin{equation}
  \label{p.property}
  p(x,t) = p^*(x) \quad \text{for} \quad t > x^2/\alpha^2 \,. 
\end{equation}
\end{itemize}
In the following, we sketch that a small modification of the existence
proof in \cite{HilhorstHM:2009:MathematicalSO} yields weak solutions
that satisfy condition (P).  This argument shows that condition (P) is
a natural additional requirement on weak solutions.  Within this
restricted class of weak solutions, non-uniqueness can only originate
from essential differences of the precipitation functions that first
occur in $D_u$ or on the parabola $\Pc$.  This is a much harder
problem and the subject of the remaining sections of this paper.

\begin{theorem} \label{t.pp}
There exists a solution $(u,p)$ to \eqref{e.original} having property
\textup{(P)}.
\end{theorem}

\begin{proof}
The proof requires a minor modification of the existence argument
given in \cite[pp.~118--123]{HilhorstHM:2009:MathematicalSO}.  Their
construction proceeds in three steps.  First, they consider the weak
formulation of a mollified version of the second-stage reaction of the
Keller--Rubinow process which, written formally in its strong form and
in non-dimensional variables, reads
\begin{equation}
  c_t = c_{xx} + k \, a_k \, b_k - c \, H_\eps
    \biggl( \int_0^t \bigl(c(x,s) - u^* \bigr)_+ \, \d s \biggr) \,,
  \label{e.kr2}
\end{equation}
where $k \, a_k \, b_k$ is the known Keller--Rubinow source term,
coming from the first-stage reaction, and $H_\eps$ is a smooth
non-decreasing approximation of the Heaviside graph with
$H_\eps (s) = H (s)$ for all $s<0$ and $s>\eps$.  This problem is
formulated as a fixed point problem for a map $\Gamma$
\cite[p.~119]{HilhorstHM:2009:MathematicalSO} which is shown to be
continuous and compact on a bounded subset $\Cc$ of the continuous
functions; existence is then a consequence of the Schauder fixed point
theorem.  Second, they let $\eps \to 0$ and extract a subsequence that
converges against a weak solution of the un-mollified version of
\eqref{e.kr2}, which corresponds to the original model of Keller and
Rubinow.  Finally, they take the fast reaction limit $k\to\infty$,
where the source term $k \, a_k \, b_k$ converges to the singular
source in \eqref{e.original.a} weakly in measure, and prove that the
corresponding sequence of Keller--Rubinow solutions has a converging
subsequence which limits to a weak solution of \eqref{e.original}.

Our goal is to enforce condition (P) across these two limits.  We
begin by modifying the second-stage reaction equation \eqref{e.kr2} to
\begin{equation}
  c_t = c_{xx} + k \, a_k \, b_k - c \, H_\eps
    \biggl( \int_0^{\min\{t,x^2/\alpha^2\}}
    \bigl(c(x,s) - u^* \bigr)_+ \, \d s \biggr) \,.
  \label{e.kr2p}
\end{equation}
The corresponding map $\Gamma$, even though it ceases to map into
$C^\infty$, remains compact from $\mathcal C$ into itself, the
relevant estimates remaining literally unchanged.  Likewise, the proof
of continuity is not affected by the change, so that the Schauder
fixed point argument applies as before.  As $\eps \to 0$, we extract a
subsequence that converges to the weak formulation of the un-mollified
version of \eqref{e.kr2p}.  The required estimates do not change and
the limit solution satisfies condition (P) by construction.

We finally reconsider the fast reaction limit
\cite[Theorem~2.7]{HilhorstHM:2009:MathematicalSO}.  The compactness
estimate remains unchanged, so that we can extract a subsequence $c_k$
which converges to a limit concentration $u$ strongly in the same
H\"older class as before.  In particular, the precipitation term
converges weakly in $L_{\text{loc}}^2(\R \times [0,T])$ to a
precipitation function $p(x,t)$ taking values in $[0,1]$ (note that
\cite{HilhorstHM:2009:MathematicalSO} use the symbol $\Chi$ in place
of $p$ here).  Moreover, $p$ is defined point-wise for every $x$, $p$
is non-decreasing in time as the limit of non-decreasing functions,
and satisfies condition (P) with
\begin{subequations}
  \label{e.modified-prec}
\begin{equation}
  p(x,t) = 1
  \qquad\text{if }
  \int_0^{^{\min\{t,x^2/\alpha^2\}}} (u(x,s) - u^*)_+ \, \d s > 0
\end{equation}
and
\begin{equation}
  p(x,t) = 0
  \qquad\text{if }
  u (x,s) < u^* \text{ for all } s \leq \min\{t,x^2/\alpha^2\} \,.
\end{equation}
\end{subequations}
The pair $(u,p)$ satisfies the weak form \eqref{weak.sol.def.eq} just
as in \cite{HilhorstHM:2009:MathematicalSO}.  It satisfies the
precipitation condition \eqref{e.hhmo-p-weak-alternative} on $D_u$ and
$\Pc$ via \eqref{e.modified-prec}.  Thus, it remains to verify that
the precipitation condition \eqref{e.hhmo-p-weak-alternative} is
satisfied on $D_o$ as well.  As $\psi$ is constant on $D_o$ and $p$ is
non-decreasing in time, a monotonicity argument, stated as
Lemma~\ref{u-psi.non-incr} below, implies that $u$ is non-increasing
in time on $D_o$; we note that the limit weak solution satisfies the
conditions of the lemma---we do not require
\eqref{e.hhmo-p-weak-alternative} to hold \emph{a priori}.  This
implies that \eqref{e.hhmo-p-weak-alternative} holds on $D_o$ as well
so that $(u,p)$ is a weak solution in the sense of
Definition~\ref{weak.sol.def}.
\end{proof}

\begin{remark}
This result does not imply that all weak solutions in the sense of
Definition~\ref{weak.sol.def} satisfy property (P), only that the
solution obtained via the modified limiting process satisfies property
(P).  Moreover, this argument does not say anything about uniqueness
of weak solutions satisfying property (P).  However, we can conclude
that non-uniqueness of solutions satisfying property (P) must
originate from differences in the precipitation function on $D_u$ or
on $\Pc$.
\end{remark}


We conclude this section with a collection of important auxiliary
results.  The first can be understood as a variation of the parabolic
maximum principle.

\begin{lemma}
\label{max-principle-sim}
Let $u$ be a weak solution to \eqref{e.original}.  Given two points
$(X,T)$ and $(x,t)$ in $D_u$ with $T>0$, $x>X$, and $t \leq T$, we
have
\begin{equation}
  \max_{s\in[0,T]}u(X,s) > u(x,t) \,.
\end{equation}
\end{lemma}

\begin{proof}
By Lemma~\ref{u.psi}, $u\le\psi$.  So we can find a point $X_1>x>X$
such that
\begin{equation}
  \max_{s\in[0,T]}u(X_1,s)\le\psi(X_1,T)<\max_{s\in[0,T]}u(X,s) \,.
  \label{e.strict1}
\end{equation}
We set $U=(X,X_1)\times(0,T)$ and denote its parabolic boundary by
$\Gamma$.  Since $D_u$ is free of sources, the maximum principle
implies
\begin{equation}
  u(x,t)\le\max_{\overline{U}} u = \max_{\Gamma} u = \max_{s\in[0,T]}u(X,s)
  \label{e.strict2}
\end{equation}
with equality if and only if $u\equiv u(x,t)$ on $[X,X_1]\times[0,t]$.
Thus, due to \eqref{e.strict1}, the inequality in \eqref{e.strict2}
must be strict.
\end{proof}

To proceed, we introduce some more notation. When $u^*<\Psi(\alpha)$,
we write $\alpha^*$ to denote the unique solution to
\begin{equation}
  \label{alpha.star}
  \Psi(\alpha^*)=u^* \,,
\end{equation}
where $\Psi$ is the precipitation-less solution given by equation
\eqref{self.similar.p0}, and we set
\begin{equation}
\label{d.star}
 D^* = \{(x,t) \colon 0<\alpha^*\sqrt t<x\} \,.
\end{equation}
We then recall two elementary properties of weak solutions whose
detailed proofs can be found in the papers cited.

\begin{lemma}[{\cite[Lemma~2]{DarbenasHO:2018:LongTA}}]
\label{u.psi}
A weak solution $(u,p)$ of \eqref{e.original} satisfies
$[u-\psi](x,0)=0$, $0<u\le\psi$ for $t>0$, and $p=0$ on $D^*$.
\end{lemma}

\begin{lemma}
\label{u-psi.non-incr} 
Suppose that $p$ is a measurable, non-negative, bounded function and
suppose $(u,p)$ satisfies the properties of a weak solution to
\eqref{e.original} except perhaps for the precipitation condition,
Definition~\ref{weak.sol.def}\ref{weak.iii}.  Then the function
$u-\psi$ is non-increasing in $t$ on $\R\times[0,T]$.
\end{lemma}

\begin{proof}
The proof stated in \cite[Lemma~3.3]{HilhorstHM:2009:MathematicalSO}
or \cite[Lemma~8]{DarbenasHO:2018:LongTA} for weak solutions applies
literally.  We note that Definition~\ref{weak.sol.def}\ref{weak.iii}
is not required in the proof and can be relaxed to the condition
stated.
\end{proof}

\begin{corollary} \label{c.ut-upper-bound}
There exists $C_\psi>0$ such that for every weak solution $(u,p)$,
\begin{equation}
  \esssup_{x \in \R} u_t(x,t)
  \leq \esssup_{x \in \R} \psi_t(x,t)
  \leq \frac{C_\psi}t \,.
  \label{e.ut-upper-bound}
\end{equation}
\end{corollary}

\begin{proof}
By direct computation, setting $z = x/\sqrt t$, we find that
\begin{equation}
  \esssup_{x \in \R} \psi_t(x,t)
  \leq \frac{\alpha\beta}{4t} \, \e^{\tfrac{\alpha^2}4} \,
       \sup_{z \in \R} z \, \e^{-\tfrac{z^2}4}
  \equiv \frac{C_\psi}t \,.  
\end{equation}
Lemma~\ref{u-psi.non-incr} implies that  $u_t\le\psi_t$ a.e., so the
claim is proved.
\end{proof}

\begin{corollary} \label{c.ut-int-upper-bound}
Let $(u,p)$ be a weak solution to \eqref{e.original}.  Then
\begin{equation}
  \int_0^t\int_{\R} \Phi(x-y,t-s) \,
    p(y,s) \, \psi_t(y,s)\,\d y\,\d s
  \leq \sqrt \pi \, \alpha^* \, C_\psi
\label{psi.der.t.int}
\end{equation}
for every $(x,t) \in \R\times\R_+$. 
\end{corollary}

\begin{proof}
Using the right-hand bound of \eqref{e.ut-upper-bound} and recalling,
from Lemma~\ref{u.psi}, that $p=0$ on $D^*$, we find that the left
hand side of \eqref{psi.der.t.int} is bounded above by
\begin{equation}
  \int_0^{t}\frac{C_\psi}{s\sqrt{4\pi(t-s)}}
        \int_{-\alpha^*\sqrt s}^{\alpha^*\sqrt s} \d y\,\d s
     \notag \\
  = \frac{\alpha^* \, C_\psi}{\sqrt{\pi}}
      \int_0^t\frac{\d s}{\sqrt{(t-s)s}}
\end{equation}
By the change of variables $s=t \, \sin^2 s'$, the right hand integral
evaluates to $\pi$.
\end{proof}

\section{The ring domain}
\label{s.local-properties}

A substantial difficulty in the analysis in the HHMO-model is the
possibility that the precipitation function may take fractional values
on sets of positive measure.  On the other hand, Lemma~\ref{one.ring}
shows that at least initially, the HHMO-solution forms a proper ring,
i.e., the precipitation function takes binary values in some bounded
region of space-time.  In this section, we introduce the \emph{ring
domain} as the maximal set of the form $\R \times (0,T^*)$ on which
$p$ is essentially binary.  On the ring domain, we are able to obtain
an elementary characterization of the precipitation boundary: we shall
show that there exists a precipitation domain $I$ and precipitation
boundary function $\ell \colon I \to \R_+$ with certain ``nice''
properties such that the precipitation function is a.e.\ given by
\begin{equation}
  \label{non-decreasing.p}
  p(x,t)
  = \begin{cases}
      \I_{\{t>\ell(x)\}}(x,t) & \text{if $x\in I$}\\
      0 & \text{otherwise} \,.
    \end{cases} 
\end{equation}

For a given, fixed weak solution $(u,p)$ of the HHMO-model
\eqref{e.original}, we write $D_u$ as the union of three disjoint
subsets,
\begin{subequations}
\begin{gather}
  P = \{ (x,t) \in D_u \colon
         u(x,s)>u^*\text{ for some }s\in[0,t]\} \,, \\
  S = \{ (x,t) \in D_u \colon
         u(x,s)<u^* \text{ for all $s\in[0,t]$}\} \,, \\
\intertext{and}
  C = \{(x,t)\in D_u \colon \max_{s\in[0,t]}u(x,s) = u^* \} \,.
\end{gather}
\end{subequations}
The set $P$ is the precipitation set where we know that $p=1$.
Likewise, $S$ is a set where precipitation cannot occur and we know
that $p=0$.  By continuity of $u$, these two sets are open.  The set
$C$ is the critical set where the precipitation threshold is reached,
but not exceeded.  In our notion of weak solution, we cannot assign a
definitive value to $p$ on $C$ but, as we shall show now, $C$ is of
measure zero.  By definition, the sections of $S$, $C$, and $P$ are
strictly ordered, i.e., for fixed $x$,
\begin{equation}
  \{ t \colon (x,t) \in S \}
  < \{ t \colon (x,t) \in C \}
  < \{ t \colon (x,t) \in P \} \,.
  \label{e.ordering}
\end{equation}

\begin{lemma}[The critical subset of $D_u$ is a null set]
\label{p.lemma}
Assume that $(u,p)$ is a weak solution to \eqref{e.original}.  
Then
\begin{enumerate}[label={\upshape(\roman*)}]
\item\label{e.5}
$C \subset \partial P$ and $C \subset \partial S$,

\item\label{e.iiii} $C$ is a set of measure zero.
\end{enumerate}
\end{lemma}

\begin{proof}
Let $(x,t) \in C$.  Then there exists $s \in (0,t]$ such that
$u(x,s) = u^*$.  By Lemma~\ref{max-principle-sim}, any point
$(X,s) \in D_u$ with $X<x$ has $\max_{s' \in [0,s]} u(X,s') > u^*$ so
that $(X,t) \in P$.  The same argument shows that $(X,t) \in S$ if
$X>x$.  Thus, $(x,t)$ is a limit point of $P$ and of $S$, which proves
\ref{e.5}.  The argument further shows that for every $t$ there is at
most one value of $x$ such that $(x,t) \in C$.  This proves
\ref{e.iiii}.
\end{proof}

The argument used in the proof of Lemma~\ref{p.lemma} cannot be
extended to critical subsets $\{(x,t) \colon u(x,t) = u^*\}$ on or
above the parabola $\Pc$.  In that case, the precipitation pattern may
be topologically complex and/or essentially non-binary.  Thus, in the
remainder of the paper we restrict ourselves to the \emph{ring
domain}, defined as follows, on which such degeneracies are not
possible.

\begin{definition}
\label{ring.domain}
We shall say that the solution $(u,p)$ to \eqref{e.original} has a \emph{ring domain}
\begin{equation}
  \RD(u)=\R\times (0,(X^*/\alpha)^2)
\end{equation}
with $X^*\in(0,+\infty]$ if there exist a strictly increasing
sequence, finite with $0=X_0<X_1<X_2<\ldots<X_n=X^*$, $n\ge1$, or infinite
with $0=X_0<X_1<X_2<\ldots<X_n<\ldots<X^*$ and
$\lim_{i\to\infty}X_i=X^*$, such that 
\begin{enumerate}[label={\upshape(\roman*)}] 
\item $[X_{2i},X_{2i+1}]$ is a ring for all applicable indices $i$,
\item $[X_{2i+1},X_{2i+2}]$ is an interring for all applicable indices
$i$,
\item when the sequence $\{X_i\}$ is finite, the interval
$[X^*,X^*+\xi]$ is neither a ring nor an interring for every
$\xi>0$.
\end{enumerate}
\end{definition}

\begin{remark}
\label{RD.indep}
Weak solutions to the HHMO-model are essentially determined by the
field $u$ alone \cite[Lemma~3]{DarbenasHO:2018:LongTA}.  This
justifies writing $\RD(u)$ instead of $\RD(u,p)$.  Below, when no
ambiguity can occur, we will often write $\RD$ for short.
\end{remark}

When the precipitation threshold is supercritical, i.e., when
$u^*<\Psi(\alpha)$, Lemma~\ref{one.ring} ensures that an initial
precipitation ring always exists, so that we can construct a
non-trivial ring domain iteratively.

We now introduce notation for three distinct parts of the
precipitation boundary,
\begin{subequations}
\begin{gather}
  \Lambda_\reg = \{ (x,t) \in C \colon (x,s) \notin C
    \text{ for } s<t \} \,, \\
  \Lambda_\deg = \{ (x,t) \in \Pc \colon x
    \text{ is contained in a ring and } (x,s) \notin C
    \text{ for } s<t\} \,, \label{e.Lambda-deg}
\intertext{and}
  \Lambda_\jump = C \setminus \Lambda_\reg \,. 
\end{gather}
\end{subequations}
We remark that, by continuity of $u$, if a line $x = \text{const}$
intersects $C$, it also intersects $\Lambda_\reg$, precisely at the
smallest value of $t$ where $\max_{s\in[0,t]}u(x,s)=u^*$.

Numerical evidence indicates that $\Lambda_\jump$ is empty and
$\Lambda_\deg$ consists only of the boundary points of $\Lambda_\reg$.
On the other hand, we have no proof that this is so.  Moreover, we
think that modifications of the model such as the addition of
non-singular loss or source terms may well create degenerate parts or
jumps in the precipitation boundary.  For this reason, we allow for
the occurrence of all three boundary components.
Figure~\ref{f.boundarysketch} illustrates the notation introduced in a
made-up sketch; we emphasize that actual numerical simulations look
different (cf.\ \cite{DarbenasO:2021:BreakdownLP}).

\begin{figure}
\centering
\includegraphics{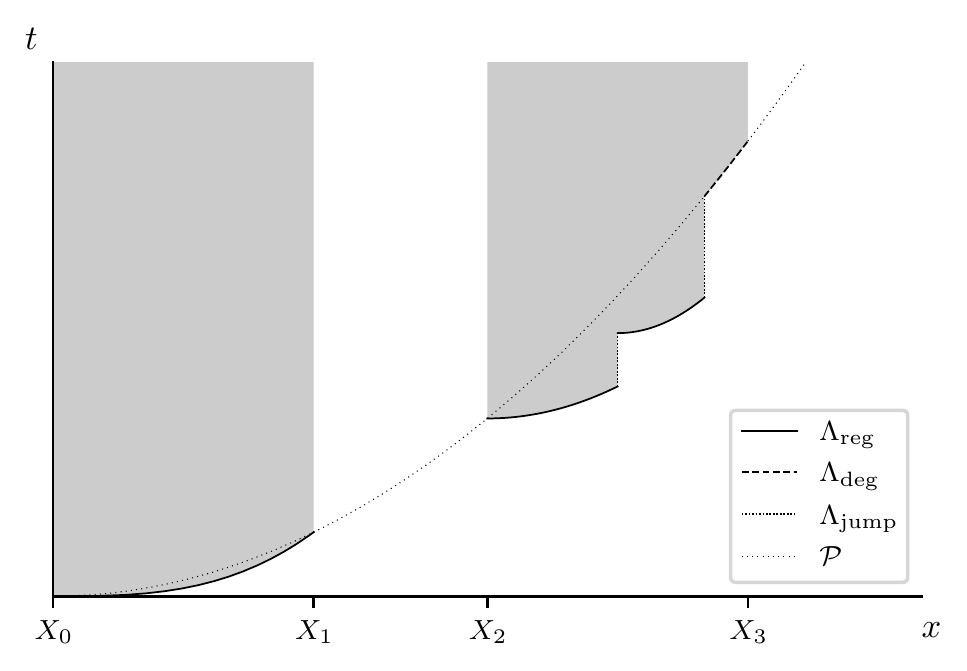}
\caption{Sketch of the three boundary components referred to in this
paper; we set $\Lambda_\normal = \Lambda_\reg \cup \Lambda_\deg$.}
\label{f.boundarysketch}
\end{figure}

To proceed, set
\begin{gather}
  \Lambda_\normal = \Lambda_\reg \cup \Lambda_\deg 
\end{gather}
and let $I$ denote the closed union of the $x$-projection of the
precipitation rings, i.e.,
\begin{gather}
  I = \bigcup_i \, [X_{2i}, X_{2i+1}] \,.
\end{gather}
By construction, whenever $x \in I$, there exists a unique $t \geq 0$
such that either $(x,t) \in \Lambda_\reg$ or $(x,t) \in \Lambda_\deg$.
Hence, we can parametrize onset of precipitation in time with a
function $\ell \colon I \to \R_+$, the \emph{precipitation front},
satisfying
\begin{equation}
  \label{l.function}
  \Lambda_\normal = \{ (x, \ell(x)) \colon x \in I \} \,.
\end{equation}

\begin{lemma}
\label{I.l.prop}
Let $(u,p)$ be a weak solution to \eqref{e.original} with ring domain
$\RD$.  Then
\begin{enumerate}[label={\upshape(\roman*)}] 
\item\label{i.prop.1} When $x\in I$, $u(x,\ell(x))=u^*$ and
$u(x,t) < u^*$ for all $t \in [0,\ell(x))$.

\item\label{i.prop.3} $\ell$ is strictly increasing and
left-continuous on $I$,

\item\label{i.prop.3a} $\ell$ is right-continuous at every $X_{2j}$
with $\ell(X_{2j})=(X_{2j}/\alpha)^2$.
\end{enumerate}
\end{lemma}

\begin{proof}
For $(x,\ell(x)) \in \Lambda_\reg$, statement \ref{i.prop.1} holds by
definition of $\Lambda_\reg \subset C$.

If $(x,\ell(x)) \in \Lambda_\deg$, we argue by contradiction.  First,
suppose $u(x,\ell(x))<u^*$.  Then there exists a neighborhood of
$\{x\}\times[0, x^2/\alpha^2]$ on which $u<u^*$ which carves out a
part of $\Pc$ that contains $(x, \ell(x))$.  Thus, $x$ is not
contained in a ring, contradicting the definition of $\Lambda_\deg$.
Else, if $u(x,\ell(x))>u^*$, there exists a neighborhood of
$(x,\ell(x))$ on which $u>u^*$.  Thus, by continuity, there exists
$t<\ell(x)$ such that $(x,t) \in \Lambda_\reg$, again contradicting
the definition of $\Lambda_\deg$.

For \ref{i.prop.3}, we first note that the argument used in the proof
of Lemma~\ref{p.lemma} applies literally and proves that $\ell$ is
strictly increasing.  Further, it is bounded, so possesses a right
limit at every point that is not a left boundary point.  Taking
$x\in I$ with $x \neq X_{2j}$ and setting
$\ell^*=\lim_{y\nearrow x}\ell(y)$, we have, by continuity of $u$,
\begin{equation}
  u(x,\ell^*)=\lim_{y\nearrow x}u(y,\ell(y))=u^* \,.
\end{equation}
This shows that $(x,\ell^*)\notin S$ so that, due to the ordering
\eqref{e.ordering}, we have $\ell^*\geq \ell(x)$.  On the other hand,
as $\ell$ is increasing, $\ell^* \leq \ell(x)$.  This proves that
$\ell^*=\ell(x)$, i.e., $\ell$ is left-continuous on $I$.

To prove \ref{i.prop.3a}, note that $0\le \ell(x)\le x^2/\alpha^2$ for
$x\in I$, so $\ell(0)=0$.  For $j>0$, suppose that
$\ell(X_{2j})<(X_{2j}/\alpha)^2$.  Then, due to
Lemma~\ref{max-principle-sim}, the precipitation condition
\eqref{e.hhmo-p-weak-alternative} is satisfied for
$x\in(\alpha\sqrt{\ell(X_{2j})},X_{2j})$, i.e.
\begin{equation}
  \max_{s\in[0,\ell(X_{2j})]}u(x,s)>u(X_{2j},\ell(X_{2j}))=u^* \,.
\end{equation}
Thus, the $j$th ring must start no farther than
$x=\alpha\sqrt{\ell(X_{2j})}$, contradiction. Thus,
$\ell(X_{2j})=(X_{2j}/\alpha)^2$ and, since $\ell$ is increasing,
$\ell(X_{2j}) \le \ell(x) \le x^2/\alpha^2$ for $x\ge X_{2j}$.  Thus
$\ell$ is right-continuous at this point.
\end{proof}

\begin{lemma}
\label{i.prop.7} 
Let $(u,p)$ be a weak solution to \eqref{e.original} with ring domain
$\RD$.  On $\RD$, $p$ can be identified, up to modification on sets
of measure zero, with
\begin{equation}
\label{p.redefined}
  p(x,t)
  = \begin{cases}
      \I_{\{t>\ell(x)\}}(x,t) & \text{if $x\in I$}\\
      0 & \text{otherwise} \,.
    \end{cases} 
\end{equation}
\end{lemma}

\begin{proof}
On $D_o \cap \RD$, the value of $p$ is determined a.e.\ by the
definition of ring domain as a sequence of rings and interrings and
agrees with \eqref{p.redefined}.  On $P$ and $S$, $p$ takes values $1$
and $0$, respectively.  Due to the ordering \eqref{e.ordering} and the
definition of $\ell$, these values also agree with
\eqref{p.redefined}.  This already suffices, because, by
Lemma~\ref{p.lemma}, the three sets $D_o \cap \RD$, $P$, and $S$ cover
the ring domain up to sets of measure zero (those being $C$, $\Pc$,
and the line $\{x=0\}$).
\end{proof}

\begin{corollary}
\label{c.point-off-lambda-normal}
In the setting of Lemma~\ref{i.prop.7}, let
$(x,t) \in \RD \setminus \Lambda_\normal$.  Then there exists a
rectangular neighborhood
$B=(x_1,x_2) \times (t_1,t_2) \subset \RD \setminus \Lambda_\normal$
of $(x,t)$ such that $p(y,s) = p^*(y)$ for all $(y,s) \in B$.
\end{corollary}

\begin{proof}
$\Lambda_\normal$ is closed, so $\RD \setminus \Lambda_\normal$ is
open.  Since, by \eqref{p.redefined} and Lemma~\ref{I.l.prop}, for
fixed $y$, $p(y,s)$ changes value only if $(y,s) \in \Lambda_\normal$,
the claim is obvious.
\end{proof}

\section{On the differentiability of $u$ and the continuity of $u_t$}
\label{sec:u.dif.cont}

In this section, we provide conditions on the existence of a classical
time derivative for the solution to the HHMO-model.  It turns out that
$u-\psi$ is always time-differentiable away from the location of onset
of precipitation.  However, time-differentiability may fail on the
precipitation boundary $\Lambda_\normal$.  In Theorem~\ref{cont.theo},
we show that time-differentiability is equivalent to continuity of the
formal time derivative.  Afterwards, in
Lemma~\ref{Lambda.normal.cond}, we present a sufficient condition:
essentially, time-differentiability holds at points where the
precipitation front is transversal to time levels $t=\text{const}$.

\begin{theorem}
\label{cont.theo}
Let $(u,p)$ be a weak solution to \eqref{e.original} with ring domain
$\RD$.  Set 
\begin{subequations}
\begin{gather}
  \Fc_1 (x,t)
  = \int_0^t\,\int_\R\Phi(x-y,t-s) \, 
        p(y,s) \, u_t(y,s) \, \d y \, \d s \,, \\
  \Fc_2(x,t) = \int_{I(u)} \Phi(x-y,t-\ell(y)) \, \d y \,.
\end{gather}
\end{subequations}
Then $u-\psi$ is differentiable in time near $(x,t) \in \RD$ and
$(u-\psi)_t$ is continuous at $(x,t)$ if and only if $\Fc_2$ is
continuous at $(x,t)$.  At a point of continuity,
\begin{equation}
  (u-\psi)_t = - \Fc_1 - u^* \, \Fc_2 \,.
  \label{u-psi.der.form}
\end{equation}
The set of points of continuity includes
$\RD \setminus \Lambda_\normal$.
\end{theorem}

\begin{remark}
In the definition of $\Fc_2$, we use the convention that
$\Phi(x-y,t-\ell(y))=0$ for $t<\ell(y)$.  In the proof, we show that $\Fc_1$
and $\Fc_2$ are well-defined even though we cannot exclude that there
are points $(x,t) \in \Lambda_\normal$ where $\Fc_1(x,t) = - \infty$
or $\Fc_2(x,t) = \infty$.
\end{remark}

\begin{remark}
The difficulty with showing that $\Fc_2$ is continuous is seen as
follows.  Suppose $\ell(y) = t - (x-y)^2$ near $y=x$.  Then
$\Phi(x-y,t-\ell(y)) = \text{const} \cdot \lvert x-y \rvert^{-1}$,
which is not integrable.  Thus, continuity of $\Fc_2$ at a boundary
point necessarily depends on the geometry of the precipitation
front.  For example, if the front advances at a non-vanishing rate,
$\Phi(x-y,t-\ell(y))$ remains integrable and continuity of $\Fc_2$
follows, e.g., by approximating the integrand by a sequence of
continuous compactly supported functions.  We discuss sufficient
conditions for continuity in Lemma~\ref{Lambda.normal.cond} and
Lemma~\ref{l.transversality2} further below.
\end{remark}

\begin{figure}
\centering
\includegraphics{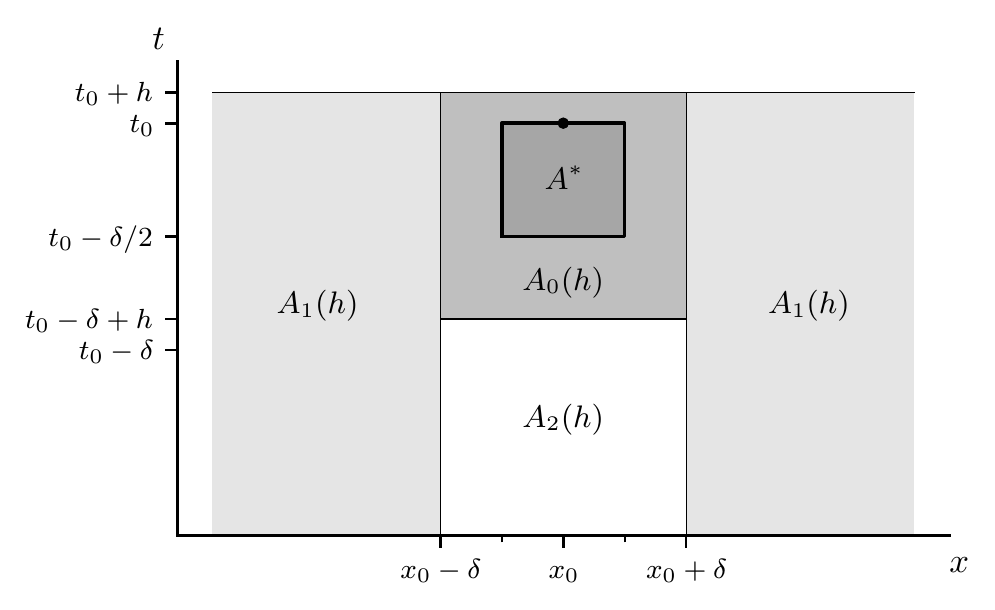}
\caption{Sketch of splitting of the domain of integration in the proof
of Theorem~\ref{cont.theo}.}
\label{f.splitting-sektch}
\end{figure}

\begin{proof}
We begin by introducing useful notation.  For any function of two
variables, $f(x,t)$, and any $h \neq 0$, we write
\begin{equation}
  \label{fin.diff.form}
  \Delta_h f(x,t) = \frac{f(x,t+h)-f(x,t)}{h} \,.
\end{equation}
For any fixed $(x_0,t_0) \in \RD \setminus \Lambda_\normal$ and
$\delta>0$, we introduce the subdomains
\begin{subequations}
  \label{e.partition}
\begin{gather}
  A_0(h) = (x_0-\delta, x_0+\delta) \times (t_0-\delta+h, t_0+h) \,,
  \\
  A_1(h) = \R \setminus (x_0-\delta, x_0+\delta) \times (0,t_0+h) \,,
  \\
  A_2(h) = (x_0-\delta, x_0+\delta) \times (0,t_0-\delta+h] \,,
\intertext{and}
  A^* = (x_0-\delta/2, x_0+\delta/2) \times (t_0-\delta/2,t_0) \,;
\end{gather}
\end{subequations}
see Figure~\ref{f.splitting-sektch}.  Due to
Corollary~\ref{c.point-off-lambda-normal}, we can choose $\delta$
sufficiently small such that for a.e.\
$y \in (x_0-\delta, x_0+\delta)$,
\begin{equation}
  p(y,s) = p^*(y)
  \label{e.pstar}
\end{equation}
for all $s \in (t_0-\tfrac54 \delta,t_0+\frac14\delta)$.  We also
choose $\delta$ sufficiently small that this time interval lies within
the temporal extent of the ring domain $\RD$.  Then for all
$\lvert h \rvert < \delta/4$, which we assume henceforth, we can use
\eqref{e.pstar} in any integral over the subregion $A_0(h)$.  The
proof now proceeds in five distinct steps.  \noqed
\end{proof}

\begin{step}
There exists a finite constant $C>0$ which may depend on the choice of
$(x_0,t_0) \in \RD \setminus \Lambda_\normal$ and $\delta$ such that
\begin{equation}
  \sup_{\substack{(x,t) \in A^* \\ \lvert h \rvert < \delta/4}} \,
  \lvert \Delta_h u (x,t) \rvert < C \,.
  \label{e.step1}
\end{equation}
\end{step}

\begin{proof}[Proof of Step 1]
First, we note that $\psi(x,t)$ is absolutely continuous in $t$ with a
uniform bound $C^*$ on $\psi_t$ where it exists and for $t$ bounded
away from zero.  Therefore, Lemma~\ref{u-psi.non-incr} implies that
for every $(x,t) \in \RD$ and $h$ small enough,
\begin{equation}
  \Delta_h u(x,t) \leq \Delta_h \psi(x,t) \leq C^* \,.
  \label{e.delta-u-upper}
\end{equation}
Thus, the main task is to find a lower bound for $\Delta_h u$.

A weak solution to the HHMO-model satisfies the Duhamel formula
\begin{equation}
  u(x,t) = \psi(x,t) - \int_0^t \int_\R \Phi(x-y, t-s) \,
    p(y,s) \, u(y,s) \, \d y \, \d s \,,
  \label{e.duhamel}
\end{equation}
see, e.g., \cite{Darbenas:2018:PhDThesis} for a detailed discussion of
the functional setting.  Fix $(x,t) \in A^*$.  Using \eqref{e.duhamel}
for each of the two terms in the finite difference $\Delta_h u(x,t)$,
breaking up the domain of integration into $A_0(h)$, $A_1(h)$, and
$A_2(h)$ for the first term and $A_0(0)$, $A_1(0)$, and $A_2(0)$ for
the second, separating out the difference between these sets of
integration, performing a ``summation by parts'' by change of
variables on the subdomain $A_0(0)$, and using \eqref{e.pstar} on the
part of the domain where it is applicable, we find that
\begin{align}
  \Delta_h u(x,t)
  & = \Delta_h \psi(x,t)
      - \iint_{A_0(0)} \Phi(x-y,t-s) \, p^*(y) \, \Delta_h u(y,s) \,
          \d y \, \d s
      \notag \\
  & \quad - \iint_{A_1(\max(0,h))} \Delta_h \Phi(x-y,t-s) \,
        p(y,s) \, u(y,s) \, \d y \, \d s
      \notag \\
  & \quad - \iint_{A_2(\min(0,h))} \Delta_h \Phi(x-y,t-s) \,
        p(y,s) \, u(y,s) \, \d y \, \d s
      \notag \\
  & \quad - \frac1h \iint_{A_2(0) \bigtriangleup A_2(h)}
        \Phi(x-y,t-s+\max(0,h)) \, 
        p^*(y) \, u(y,s) \, \d y \, \d s
      \notag \\
  & \equiv \Delta_h \psi(x,t) + \Ic_0 + \Ic_1 + \Ic_2 + \Ic_2^* \,,
  \label{e.delta-u-split}
\end{align}
where $A \bigtriangleup B =(A \setminus B) \cup (B \setminus A)$
denotes the symmetric difference of the sets $A$ and $B$.  We remark
that we only need to account for the symmetric difference between the
sets $A_2(0)$ and $A_2(h)$; the other symmetric differences are
implicit via the convention that $\Phi(x-y,t-s)=0$ for $s>t$.

The first term on the right of \eqref{e.delta-u-split} is bounded
below by uniform absolute continuity of $\psi$ as before.  Next, due
to \eqref{e.delta-u-upper},
\begin{equation}
  \Ic_0 \geq - C^* \int_0^{t_0} \int_\R \Phi(x-y,t-s) \, \d y \, \d s
      = - C^* \, t_0 \,.
\end{equation}
To proceed, recall that $u \leq \Psi(\alpha)$ by Lemma~\ref{u.psi} and
note that the effective horizontal domain of integration is bounded.
Moreover,
\begin{equation}
  \sup_{\substack{(y,s) \in A_1(\max(0,h)) \\
    (x,t) \in A^*}} \, \Delta_h \Phi(x-y,t-s)
  \leq \sup_{\substack{\lvert y \rvert \geq \delta/4 \\
    t \geq \delta/4}} \, \Phi_t(y,s)
  \label{e.delta-psi-estimate}
\end{equation}
is bounded.  This provides the lower bound for $\Ic_1$; an analogous
argument is made for $\Ic_2$.  

Finally, we note that $m(A_2(0) \bigtriangleup A_2(h)) = 2 \delta h$.
Moreover, as in \eqref{e.delta-psi-estimate}, the singularity of the
heat kernel is at least a distance $\delta/4$ away from the domain of
integration whenever $(x,t) \in A^*$.  Thus, $\Ic_2^*$ is also bounded
below.  This concludes the proof of Step~1.
\end{proof}

\begin{step}
The function $u-\psi$ is time-differentiable at
$(x_0,t_0) \in \RD \setminus \Lambda_\normal$ with
\begin{align}
  (u-\psi)_t(x_0,t_0)
  & = - \int_{x_0-\delta}^{x_0+\delta}
        \int_{t_0-\delta}^{t_0} \Phi(x_0-y,t_0-s) \,
          u_t(y,s) \, \d s \, p^*(y) \, \d y
      \notag \\
  & \quad - \iint_{A_0(0)^c} \Phi_t(x_0-y,t_0-s) \,
          p(y,s) \, u(y,s) \, \d y \, \d s
      \notag \\
  & \quad - \int_{x_0-\delta}^{x_0+\delta} \Phi(x_0-y,\delta) \,
        p^*(y) \, u(y,t_0-\delta) \, \d y
  \label{u-psi.der.temp}
\end{align}
for some $\delta >0$.
\end{step}

\begin{proof}[Proof of Step 2]
We employ the domain partition \eqref{e.partition} with $\delta$
reduced to half its value from Step~1.  Formula
\eqref{e.delta-u-split} remains valid on this new partition.  We fix
$(x,t) = (x_0,t_0)$ and pass to the limit $h\to 0$ in each of the
terms on its right hand side as follows.

On $A_0$, Step~1 implies that for every fixed
$y \in (x_0-\delta,x_0+\delta)$, $u(y,s)$ is absolutely continuous as
a function of $s$ on the interval $(t_0-\delta,t_0)$ and therefore
differentiable a.e.\ in time with $\lvert u_t \rvert \leq C$.  Hence,
by the dominated convergence theorem,
\begin{equation}
  \lim_{h \to 0} \int_{t_0-\delta}^{t_0} \Phi(x_0-y,t_0-s) \,
    \Delta_h u(y,s) \, \d s
  = \int_{t_0-\delta}^{t_0} \Phi(x_0-y,t_0-s) \,
    u_t(y,s) \, \d s \,.
\end{equation}
A second application of the dominated convergence theorem, using
\begin{equation}
  y \mapsto C \int_{t_0-\delta}^{t_0} \Phi(x_0-y,t_0-s) \, \d s
\end{equation}
as the dominating function, then establishes that $\Ic_0$ converges to
the first term on the right of \eqref{u-psi.der.temp}.

For the remaining terms, due to the boundedness of $u$, $\Phi$, and
$\Phi_t$, we invoke the dominated convergence theorem directly to
establish convergence to the corresponding terms on the right of
\eqref{u-psi.der.temp}.
\end{proof}

\begin{remark}
In the proof of Step~2, Borel-measurability of $u_t$ on $A_0$ is not
easily asserted so that we claim the first term on the right of
\eqref{u-psi.der.temp} only in the sense of iterated partial
integrals.  However, once \eqref{u-psi.der.temp} is established,
measurability in two dimensions is obvious \emph{a posteriori}; see
Step~3 below.
\end{remark}

\begin{step}
$(u-\psi)_t$ satisfies \eqref{u-psi.der.form} on
$\RD \setminus \Lambda_\normal$.
\end{step}

\begin{proof}[Proof of Step 3]
Step~2 shows that $u-\psi$ is time-differentiable on
$\RD \setminus \Lambda_\normal$.  In particular, $u_t$ exists a.e.\ on
$\RD$ and is measurable as the pointwise limit of the measurable
function $\Delta_hu$.  We can thus revisit the limit of $\Ic_0$,
applying the dominated convergence theorem directly on the subdomain
$A_0(0)$.  This proves that
\begin{equation}
  \lim_{h\to 0} \Ic_0
  = \iint_{A_0(0)}
    \Phi(x_0-y,t_0-s) \, p^*(y) \, u_t(y,s) \, \d s \, \d y \,.
  \label{fub.th.A0}
\end{equation}

To rewrite the remaining terms in \eqref{u-psi.der.temp}, we note once
again that $u_t$ is measurable and consider the integral
\begin{equation}
  \Ic_0^* = \iint_{A_0(0)^c} \Phi(x_0-y,t_0-s) \,
    p(y,s) \, u_t(y,s) \, \d y \, \d s \,.
  \label{e.Ic0star}
\end{equation}
By Corollary~\ref{c.ut-upper-bound} and~\ref{c.ut-int-upper-bound},
the integrand in this expression has an integrable upper bound.  Thus,
we can apply the Fubini theorem to the positive part of the integrand
and the Tonelli theorem to the negative part, to write
\begin{equation}
  \Ic_0^*
  = \int_{I^*} \int_{\ell(y)}^{c(y)} \Phi(x_0-y,t_0-s) \, u_t(y,s) \,
    \d s \, \d y \,,
  \label{fub.th.A123}
\end{equation}
where $c(y) = t_0 - \delta$ for $y \in (x_0-\delta,x_0+\delta)$ and
$c(y)=t_0$ otherwise, and
\begin{equation}
  I^* = \{ x \in I \colon \ell(x) < c(x) \} \,.
\end{equation}
Then, for $y\in I^*$, we have
\begin{align}
  \int_{\ell(y)}^{c(y)}
  & \bigl(
      \Phi(x_0-y,t_0-s) \, u_t(y,s) - \Phi_t(x_0-y,t_0-s) \, u(y,s)
    \bigr) \, \d s
    \notag \\
  & = \int_{\ell(y)}^{c(y)} \frac\partial{\partial s}
        \bigl( \Phi(x_0-y,t_0-s) \, u(y,s) \bigr) \, \d s
    \notag \\
  & = \Phi(x_0-y,t_0-c(y)) \, u(y,c(y))
      - \Phi(x_0-y,t_0-\ell(y)) \, u(y,\ell(y)) \,.
  \label{fund.th.cal}
\end{align}
Noting that $u(y,\ell(y)) = u^*$ and $\Phi(x_0-y,t_0-c(y))=0$ outside of
$y \in (x_0-\delta,x_0+\delta)$, then combining
\eqref{u-psi.der.temp}, \eqref{fub.th.A123}, and \eqref{fund.th.cal},
we find that the expression from Step~2 implies
\eqref{u-psi.der.form}.
\end{proof}

\begin{step}
Suppose that $\Fc_2$ is continuous at $(x_0,t_0) \in \RD$.  Then there
exists a neighborhood $V$ of $(x_0,t_0)$ such that $u-\psi$ is
differentiable in time on $V$, $(u-\psi)_t$ is continuous at
$(x_0,t_0)$, and \eqref{u-psi.der.form} holds at this point.
\end{step}

\begin{proof}[Proof of Step 4]
By continuity of $\Fc_2$, there exists an open neighborhood
$V \subset \RD$, of $(x_0,t_0)$, bounded away from $t=0$, such
that $\Fc_2$ is uniformly bounded on $V$.  First, we show that $u_t$
is essentially bounded on $V$.  Indeed, an upper bound is already
given by Corollary~\ref{c.ut-upper-bound}.

To obtain a lower bound, notice that, by Step~3 and
Lemma~\ref{u-psi.non-incr},
\begin{align}
  u_t
  & = \psi_t - \Fc_1 - u^* \, \Fc_2
      \notag \\
  & \geq \psi_t - \int_0^t \int_\R \Phi(x-y,t-s) \, p(y,s) \,
         \psi_t(y,s) \, \d y \, \d s
         - u^* \, \sup_{(y,s) \in V} \Fc_2(y,s)
\end{align}
a.e.\ on $V$.  The first term on the right is clearly finite on $V$,
the second by Corollary~\ref{c.ut-int-upper-bound}, and the last term
is finite by construction.

Second, we show that there exists $\delta >0$ such that
$p\, \lvert u_t \rvert$ is integrable on
$A = I \times (0,t_0+\delta)$.  To see this, fix $\delta>0$ such that
there exists $(x,t) \in V \setminus \Lambda_\normal$ with
$t>t_0+\delta$ such that $(u-\psi)_t \leq 0$ exists at this point.
Let $\sigma_- = - \min\{p u_t,0\}$ and $\sigma_+ = \max\{p u_t,0\}$
denote the negative and positive parts of $pu_t$, respectively.  By
\eqref{e.ut-upper-bound}, $\sigma_+$ is essentially bounded.  For
$\sigma_-$, we estimate
\begin{align}
  \inf_{(y,s) \in A} \,
  & \Phi(x-y,t-s) \iint_A \sigma_-(y,s) \, \d y \, \d s
    \notag \\
  & \leq \int_0^{t} \int_\R \Phi(x-y,t-s) \,
      \sigma_-(y,s) \, \d y \, \d s
    \notag \\
  & = (u-\psi)_t + \Fc_1^+ + u^* \, \Fc_2 \,, 
  \label{e.bound-neg-part}
\end{align}
where
\begin{equation}
  \Fc_1^+ = \int_0^{t} \int_\R \Phi(x-y,t-s) \,
      \sigma_+(y,s) \, \d y \, \d s 
\end{equation}
is bounded due to Corollary~\ref{c.ut-upper-bound}
and~\ref{c.ut-int-upper-bound}.  Since $\Phi(x-y,t-s)$ has a positive
lower bound on $A$ and all terms on the right hand side of
\eqref{e.bound-neg-part} are bounded, $\sigma_-$ is integrable on $A$,
and so is $p \, \lvert u_t \rvert$.

Now, for every $(x,t) \in A$,
\begin{align}
  \Fc_1(x,t)
  & = \iint_{A \setminus V}
         \Phi(x-y,t-s) \, p(y,s) \, u_t(y,s) \, \d y \, \d s
         \notag \\
  & \quad + \iint_{A \cap V}
         \Phi(x-y,t-s) \, p(y,s) \, u_t(y,s) \, \d y \, \d s \,.
  \label{e.continuity-estimate}
\end{align}
The first term is continuous by the dominated convergence theorem as
$p \, \lvert u_t \rvert$ is integrable and the kernel is bounded on
$A\setminus V$ uniformly for $(x,t)$ near $(x_0,t_0)$.  The second
term is bounded as a convolution of an $L^1$ with an $L^\infty$
function as $u_t$ is bounded on $V$.

When $(x_0,t_0) \in \RD \setminus \Lambda_\normal$, the claim follows
directly from formula \eqref{u-psi.der.form} proved in Step~3.  When
$(x_0,t_0) \in \RD \cap \Lambda_\normal$, we note that
\eqref{u-psi.der.form} holds for $x=x_0$ fixed and a.e.\ $t$ near
$t_0$ and the right hand side of \eqref{u-psi.der.form} is continuous
at $(x_0,t_0)$.  Hence, we can use \eqref{u-psi.der.form} to
continuously extend $(u-\psi)_t$ to the point $(x_0,t_0)$.
\end{proof}

\begin{step}
Suppose that there exists an open neighborhood $V$ of
$(x_0,t_0)\in \RD$ such that $u-\psi$ is differentiable in time on $V$
and $(u-\psi)_t$ is continuous at $(x_0,t_0)$.  Then $\Fc_2$ is
continuous at $(x_0,t_0)$ and \eqref{u-psi.der.form} holds at this
point.
\end{step}

\begin{proof}[Proof of Step 5]
Since $(u-\psi)_t$ is continuous at $(x_0,t_0)$, $u_t$ exists a.e.\
and is essentially bounded on a possibly smaller neighborhood, again
denoted $V$.  Following the proof of Step~4 starting from the second
claim we find, as before, that $\Fc_1$ is continuous at $(x_0,t_0)$.
Turning to $\Fc_2$, we first show that $\Fc_2$ is well-defined on $V$.
Indeed, on $V\setminus \Lambda_\normal$, the integrand is bounded, so
$\Fc_2$ is finite.  Now take $(x,t) \in V \cap \Lambda_\normal$.
Since $\ell$ is strictly increasing,
$(x+\eps,t) \notin \Lambda_\normal$ for every $\eps>0$ and
\eqref{u-psi.der.form} holds true at every such point.  Moreover,
$\ell(y)>t$ for $y>x$, so that $\Phi(x-y,t-\ell(y))$ can only be
nonzero if $y<x$ so that, for fixed $y$, $\Phi(x+\eps-y,t-\ell(y))$ is
a decreasing function of $\eps$.  Consequently, by the monotone
convergence theorem,
\begin{equation}
  \lim_{\eps \searrow 0} \Fc_2(x+\eps,t) = \Fc_2(x,t)
\end{equation}
either as a finite limit or diverging to $+\infty$.  Further, taking
$\limsup_{\eps \searrow 0}$ of \eqref{u-psi.der.form},
\begin{equation}
  \limsup_{\eps \searrow 0} (u-\Psi)_t(x+\eps,t)
  = - \Fc_1(x,t) - u_* \, \Fc_2(x,t) \,.
\end{equation}
Since the first two terms are finite, so is $\Fc_2(x,t)$.  Finally, at
the point $(x_0,t_0)$,
\begin{align}
  \limsup_{x \to x_0} \limsup_{\eps \searrow 0} \, (u-\psi)_t(x+\eps,t)
  & = \liminf_{x \to x_0} \, \limsup_{\eps \searrow 0}
      (u-\psi)_t(x+\eps,t)
      \notag \\ 
  & = (u-\psi)_t(x_0,t_0) \,.
\end{align}
Using continuity of $\Fc_1$ once again, we conclude that $\Fc_2$ is
continuous at $(x_0,t_0)$.
\end{proof}

Since $\Fc_2$ is clearly continuous for every
$(x,t) \notin \Lambda_\normal$ by dominated convergence, these five
steps conclude the proof of Theorem~\ref{cont.theo}.  On
$\Lambda_\normal$, the following lemma provides a sufficient condition
for continuity.

\begin{lemma}
\label{Lambda.normal.cond}
Let $(u,p)$ be a weak solution to \eqref{e.original} with ring domain
$\RD$.  Suppose that $(x_0,t_0) \in \Lambda_{\normal} \cap \RD$. Then
$\Fc_2$, defined in Theorem~\ref{cont.theo}, is continuous near
$(x_0,t_0)$ provided
\begin{equation}
  \label{F.2.cont.cond}
  u_{x+}(x_0,t_0)
  = \lim_{h\searrow0}\frac{u(x_0+h,t_0)-u(x_0,t_0)}h<0 \,.
\end{equation}
\end{lemma}

\begin{figure}
\centering
\includegraphics{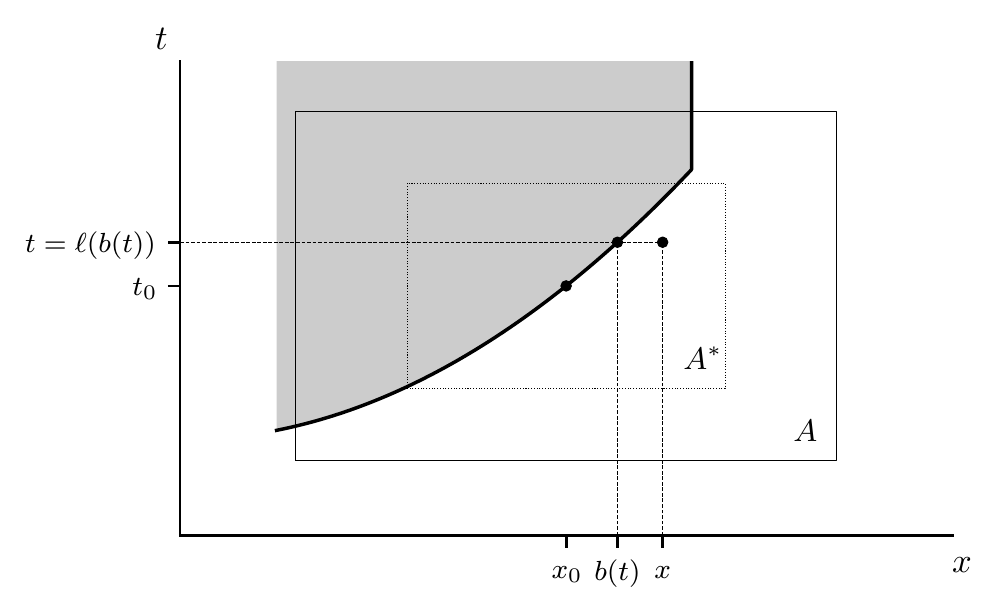}
\caption{Sketch of the geometry of the construction used in the proof
of Lemma~\ref{Lambda.normal.cond}.}
\label{f.splitting-sektch2}
\end{figure}

\begin{proof}
Since $u - \psi \in C^{1,0}(\RD)$, there exists a rectangular
neighborhood $A$ of $(x_0,t_0)$ such that
$u_{x+} < \frac12 \, u_{x+}(x_0,t_0)<0$ on $A\setminus D_o$.  We
choose $A$ small enough so that it is contained in the first quadrant,
is bounded away from the $x$-axis, and intersects only one ring.  Let
$A^*$ denote a smaller neighborhood of $(x_0,t_0)$, strictly nested
inside of $A$ (see Figure~\ref{f.splitting-sektch2}).

Writing
\begin{equation}
  I_A = \{ x \in I \colon (x,t) \in A \text{ for some } t \} \,,
\end{equation}
we split the domain of integration in the definition of $\Fc_2$ into
\begin{equation}
  \Fc_2(x,t)
  = \int_{I\setminus I_A} \Phi(x-y,t-\ell(y)) \, \d y
  + \int_{I_A} \Phi(x-y,t-\ell(y)) \, \d y \,.
  \label{e.fc2-splitting}
\end{equation}
The first term is continuous on $A^*$ by dominated convergence because
the singularity of the kernel is bounded by a uniform distance away
from the domain of integration.  Thus, the main task is to prove that
the second term is continuous on $A^*$ as well.  

We employ the Vitali convergence theorem (e.g.\
\cite{Folland:1999:RealA}).  First, we show that
\begin{equation}
  \Phi(x - \, \cdot \,, t - \ell(\,\cdot\,))
  \to \Phi(x_0 - \, \cdot \,, t_0 - \ell(\,\cdot\,))
\end{equation}
in measure as $(x,t) \to (x_0,t_0)$.  Indeed, let $\eps>0$.   For
every $r>0$, take an arbitrary $y$ with $\lvert x_0 - y \rvert > r$.
Then
\begin{align}
  \lvert \Phi(x_0&-y,t_0-\ell(y)) - \Phi(x-y,t-\ell(y)) \rvert
    \notag \\
  & \leq \lvert \Phi(x_0-y,t_0-\ell(y)) - \Phi(x_0-y,t-\ell(y)) \rvert
    \notag \\
  & \quad 
       + \lvert \Phi(x_0-y,t-\ell(y)) - \Phi(x-y,t-\ell(y)) \rvert
    \notag \\
  & \leq \sup_{y \notin B(x_0,r)}
         \sup_{\tau \in [t_0 - \ell(y), t-\ell(y)]} \,
         \lvert \Phi_t(x_0-y,\tau) \rvert \, \lvert t_0-t \rvert
    \notag \\
  & \quad + \sup_{y \notin B(x_0,r)}
         \sup_{\xi \in [x_0-y,x-y]}
         \sup_{s \in [0, \ell(y)]} \,
         \lvert \Phi_x(\xi,s) \rvert \, \lvert x - x_0 \rvert \,.
  \label{e.conv-measure-estimate}
\end{align}
The suprema on the right hand side are both finite (but may depend on
$r$).  Therefore, it is possible to choose $\delta>0$ small enough so
that the right hand side of \eqref{e.conv-measure-estimate} is less
than $\eps$ whenever $\lvert t_0-t \rvert< \delta$ and
$\lvert x - x_0 \rvert < \delta$.  Thus,
\begin{equation}
  m\{y \in \R \colon \lvert \Phi(x_0-y,t_0-\ell(y)) - \Phi(x-y,t-\ell(y))
  \rvert \geq \eps \} < 2r \,.
\end{equation}
Since $r$ was arbitrary, this proves that
\begin{equation}
  \lim_{(x,t)\to (x_0,t_0)}
  m\{y \in \R \colon \lvert \Phi(x_0-y,t_0-\ell(y)) -
  \Phi(x-y,t-\ell(y)) 
  \rvert \geq \eps \} = 0 \,,
\end{equation}
i.e., convergence in measure.

Second, we show that $\Phi(x - \, \cdot \,, t - \ell(\,\cdot\,))$ is
uniformly integrable for $(x,t) \in A^*$.  Here, it suffices to bound
the integrand by a translate of a fixed integrable profile.  Recalling
that, by Lemma~\ref{I.l.prop}, $u(y,\ell(y))=u^*$ for all $y\in I$ and
$\ell$ is strictly increasing, we find that, for $y_1,y_2\in I_A$ with
$y_1<y_2$,
\begin{align}
\label{l.deriv}
  0 & =u(y_2,\ell(y_2))-u(y_1,\ell(y_1))
      \notag \\
    & =u(y_2,\ell(y_2))-u(y_2,\ell(y_1))+u(y_2,\ell(y_1))-u(y_1,\ell(y_1))
      \notag\\
    & \le \psi(y_2,\ell(y_2))-\psi(y_2,\ell(y_1))
          +u_x(\xi,\ell(y_1)) \, (y_2-y_1)\notag\\
    & = \psi_t(y_2,\tau) \, (\ell(y_2)-\ell(y_1))
        + u_x(\xi,\ell(y_1)) \, (y_2-y_1) \,.
\end{align}
The inequality in the third line is due to Lemma~\ref{u-psi.non-incr}
which states that $u-\psi$ is non-increasing in time.  Further, we
used the mean value theorem twice, for some $\xi\in(y_1,y_2)$ and
$\tau\in(\ell(y_1),\ell(y_2))$.  We conclude that
\begin{equation}
  \frac{\ell(y_2)-\ell(y_1)}{y_2-y_1}
  \ge - \frac{u_x(x_0,t_0)}{2 \, \sup_A\psi_t}
  \equiv C_A > 0 \,.
  \label{e.transversality1}
\end{equation}
In the following, take any $(x_0,t_0) \in A^* \cap \Lambda_\normal$,
fix $(x,t) \in A^*$, and suppose that the ring intersecting $A$
intersects time-level $t$ within the interior of $A$.  (If not,
$\Phi(x - \, \cdot \,, t - \ell(\,\cdot\,))$ is essentially zero on
$I_A$ and there is nothing to do.)  Then for all $y, y_2 \in I_A$ with
$y < y_2$ such that $\ell(y_2) \leq t$ we have, by
\eqref{e.transversality1},
\begin{equation}
  t - \ell(y) \geq \ell(y_2) - \ell(y) \geq C_A \, (y_2-y) 
\end{equation}
so that
\begin{equation}
  t - \ell(y) \geq C_A \, (b(t) - y)
\end{equation}
where $b(t) = \sup \{ y \in I_A \colon \ell(y) \leq t \}$, see
Figure~\ref{f.splitting-sektch2}.  Hence,
\begin{equation}
  \Phi(x - y, t - \ell(y))
  \leq \I_{y \leq b(t)} \, \frac1{\sqrt{4 \pi \, C_A \, (b(t)-y)}} \,,
\end{equation}
which, as a translate of a fixed profile, is uniformly integrable on $I_A$.

Finally, we note that the interval of integration is bounded, so that
that the family $\Phi(x - \, \cdot \,, t - \ell(\,\cdot\,))$
restricted to $I_A$ is trivially tight.  We conclude that the Vitali
convergence theorem applies and proves that the second integral in
\eqref{e.fc2-splitting} is continuous at $(x_0,t_0)$ as well.
\end{proof}

\begin{remark}
If we think of $\Fc_2$ being defined with a general function $\ell(y)$
that does not necessarily come from the HHMO-model, there are two
failure modes for the continuity of $\Fc_2$.  The first is
topological: if the number of intersections of $\ell$ with horizontal
lines in the $x$-$t$ plane changes, the value of $\Fc_2$ can jump as
$t$ is varied.  In our setting, this is prevented by the strict
monotonicity of $\ell$.  The second failure mode is analytical: if
$\ell(y)$ crosses time-level $t$ at the wrong rate, then the integral
may diverge.  This is illustrated by the family of functions
$\ell(x) = t_0 - (x-x_0) \, \lvert x-x_0 \rvert^\gamma$.  When
$\gamma=1$, the integral diverges, whereas for any $\gamma \in (-1,1)$
or $\gamma>1$, the integral is finite.  In our setting, divergence is
prevented by the transversality condition \eqref{F.2.cont.cond} which,
as this discussion shows, is sufficient but clearly not necessary.
\end{remark}

\section{On uniqueness of the solutions}
\label{the.uniqueness}

In the following, we prove two uniqueness theorems.  The first,
Theorem~\ref{u.uniqueness}, asserts unconditional uniqueness of the
solution to the HHMO-model for a short but positive interval of time.
The second result, Theorem~\ref{t.cond-uniqueness} proves uniqueness
within the ring domain of the solution and subject to some regularity
of the precipitation front, which can be expressed as transversality
in time of the increase of concentration at the location of the front.
The proof also shows that any breakdown of uniqueness must be
accompanied by topologically complex behavior of the associated
precipitation fronts.

\begin{theorem}[Short-time uniqueness]
\label{u.uniqueness}
Assume that $u^*$ is a supercritical precipitation threshold.  Then
there exists a time $T_{\unique}>0$ such that any two weak solutions
to \eqref{e.original} are identical on
$D_{\unique}=\R\times[0,T_{\unique}]$.
\end{theorem}

\begin{proof}
A weak solution to \eqref{e.original} has at least one ring with a
width of at least $\ringwidth$, see Remark~\ref{one.ring}.  Moreover,
ignition of precipitation can appear only on some restricted domain,
the \emph{essential domain}
\begin{equation}
  \ES(t)=\{ (y,s) \colon \alpha\sqrt s < y <\alpha^{*}\sqrt s,
            0 < s < t\} \,,
\end{equation}
where $\alpha^*$ is defined by \eqref{alpha.star}.  The key step
in this proof is to show that there exists a positive time
$T_{\unique}$ such that $u_t>0$ on $\ES(T_{\unique})$ for any weak
solution $(u,p)$.  Once this is established, uniqueness up to time
$T_{\unique}$ follows by standard energy estimates.

First, we establish a negative upper bound for $u_x$.  Differentiating
the Duhamel formula \eqref{e.duhamel}, we obtain
\begin{align}
  u_x(x,t)
 & = \psi_x(x,t)
     - \int_0^t\int_{\R}\Phi_x(x-y,t-s) \, p(y,s) \, u(y,s) \, \d y\,\d s
     \notag\\
 & \le \psi_x(x,t)+\Psi(\alpha)\int_0^t\int_{\{\abs{y}\le
       \alpha^*\sqrt{s}\}}\abs{\Phi_x(x-y,t-s)} \, \d y\,\d s \,.
\end{align}
By direct computation,
\begin{equation}
  \psi_x(x,t) 
  \le - \frac{\alpha\beta}{2\sqrt t} \, 
        \e^{\tfrac{\alpha^2-\alpha^{*2}}4}
\end{equation}
on $\ES(t)$.  Further, since $\Phi_x(x,t) \in L^1(\R\times[0,T])$ for
all $T>0$, we observe that
\begin{equation}
  \lim_{t\searrow0} 
    \int_0^t\int_{\R}|\Phi_x(x-y,t-s)| \, \d y\,\d s = 0 \,.
\end{equation}
Thus, there exists $T_1>0$ such that, on $\ES(T_1)$, every weak
solution satisfies
\begin{equation}
  \label{T_1.def}
  u_x(x,t)
  \le - \frac{\alpha\beta}{4\sqrt t} \, 
        \e^{\tfrac{\alpha^2-\alpha^{*2}}4} \,.
\end{equation}
By Lemma~\ref{Lambda.normal.cond} together with
Theorem~\ref{cont.theo}, this implies that $(u-\psi)_t$ exists and is
given by \eqref{u-psi.der.form} on $\ES(T_1)$.

Second, we establish a lower bound on the growth of $\ell$.  We know
from Lemma~\ref{I.l.prop} \ref{i.prop.3} that $\ell$ is increasing on
$\R_+$.  By the Lebesgue differentiation theorem for monotonic
functions, $\ell$ is differentiable almost everywhere on
$[0,\ringwidth]$.  We denote the domain of differentiability by $U$.
Then, e.g.\ \cite[p.~108]{Folland:1999:RealA},
\begin{equation}
  \ell(y_2)-\ell(y_1)\ge\int_{[y_1,y_2]\cap U}\ell'(y) \, \d y
  \label{e.monotonic-ftc}
\end{equation} 
for all $0<y_1<y_2\le\ringwidth$.  Assuming that
$y \in (0,\ringwidth] \cap U$ with $\ell(y) \le T_1$, a computation
analogous to \eqref{l.deriv} yields
\begin{equation}
\label{l.der}
  \ell'(y)\ge -\frac{u_x(y,\ell(y))}{\psi_t(y,\ell(y))} \,.
\end{equation}
We also observe that, due to Lemma~\ref{u.psi},
\begin{equation}
  \label{l.lower.bound}
  \ell(y) \ge (y/\alpha^{*})^2 \,.
\end{equation}
Inserting \eqref{T_1.def} and \eqref{e.ut-upper-bound} into
\eqref{l.der}, then using \eqref{l.lower.bound} in a second step, we
estimate
\begin{align}
  \ell'(y)
  & \ge \biggl(\frac{C_\psi}{\ell(y)} \biggr)^{-1}
        \frac{\alpha\beta}{4\sqrt{\ell(y)}} \,
               \e^{\tfrac{\alpha^2-\alpha^{*2}}4}
    =   \frac{\alpha\beta}{4C_\psi} \, 
               \e^{\tfrac{\alpha^2-\alpha^{*2}}4} \,
        \sqrt{\ell(y)}
      \notag \\
  & \ge \frac{\alpha\beta}{4\alpha^* C_\psi} \, 
        \e^{\tfrac{\alpha^2-\alpha^{*2}}4} \, y
    =   2 \, C_\ell \, y
  \label{l.dev.bound}
\end{align}
with a constant $C_\ell$ which is independent of the weak solution
$(u,p)$.  Integrating \eqref{l.dev.bound} and recalling
\eqref{e.monotonic-ftc}, we obtain
\begin{equation}
  \ell(y_2)-\ell(y_1) \ge C_\ell \, (y_2^2-y_1^2) \,.
  \label{l.dev.bound-integrated}
\end{equation} 

Third, we obtain upper bounds on $\Fc_1$ and $\Fc_2$, hence, a lower
bound on $u_t$.  For $\Fc_1$, we estimate, invoking
Lemma~\ref{u-psi.non-incr} and Corollary~\ref{c.ut-int-upper-bound},
that
\begin{align}
  \Fc_1(x,t)
  & = \int_0^t \int_\R
      \Phi(x-y,t-s) \, p(u,s) \, u_t(y,s) \, \d y \, \d s 
      \notag \\
  & \le \int_0^t \int_{\R}
      \Phi(x-y,t-s) \, p(y,s) \, \psi_t(y,s) \, \d y \, \d s
      \notag \\ 
  & \le \alpha^* \, C_\psi \, \sqrt\pi \,. 
  \label{e.f1-upper-bound}
\end{align}
For $\Fc_2$, we restrict final time to
$T_2 = \min \{(\ringwidth/\alpha^*)^2, T_1 \}$.  Clearly, $T_2$ is
positive, independent of the weak solution $(u,p)$, and
\begin{equation}
  \label{T.star}
  \ell(\ringwidth) \ge T_2 \,.
\end{equation}
Setting
\begin{equation}
  a \equiv a(t) = \sup \{ y \colon \ell(y) \leq t \} 
\end{equation}
so that $a(t) \le \ringwidth$ and $\ell(a(t)) \le t$ due to the left-continuity
of $\ell$, see Lemma~\ref{I.l.prop}\ref{i.prop.3}. Using
\eqref{l.dev.bound-integrated}, we find that
\begin{equation}
  t-\ell(y) \ge \ell(a(t)) - \ell(y) \ge C_\ell \, (a^2-y^2)
\end{equation}
for all $y$ with $\ell(y)\le t$.  Thus, for all $x \in \R$ and
$t \in [0,T_2]$
\begin{align}
  \Fc_2(x,t)
  & = \int_I\Phi(x-y,t-\ell(y))\,\d y
    \le \frac1{\sqrt{4\pi}}\int_{-a}^a(t-\ell(y))^{-\tfrac12} \, \d y
    \notag\\
  & = \frac1{\sqrt{\pi}}\int_0^{a}(t-\ell(y))^{-\tfrac12} \, \d y
    \le \frac1{\sqrt{\pi C_\ell}} \int_0^a(a^2-y^2)^{-\tfrac12} \, \d y
    \notag\\
  & = \frac1{\sqrt{\pi C_\ell}} \, \sin^{-1} 
      \Bigl(\frac ya \Bigr)\bigg|^a_0
    = \frac12 \, \sqrt\frac{\pi}{C_\ell} \,.
  \label{e.f2-upper-bound}
\end{align}
On $\ES(T_2)$, we also have a lower bound on $\psi_t$,
\begin{equation}
  \psi_t(x,t) \ge \frac{c_\psi}t
  \label{e.psi-lower-bound}
\end{equation}
with
\begin{equation}
  c_\psi 
  = \frac{\alpha\beta}4 \, \e^{\tfrac{\alpha^2}4} \, 
    \min_{y\in[\alpha,\alpha^*]} y \, \e^{\tfrac{-y^2}4}
  > 0 \,.
\end{equation}
Altogether, inserting the bounds \eqref{e.f1-upper-bound},
\eqref{e.f2-upper-bound}, and \eqref{e.psi-lower-bound} into
\eqref{u-psi.der.form}, we obtain
\begin{align}
  u_t(x,t)
  & = \psi_t(x,t)-\Fc_1(x,t)-u^* \, \Fc_2(x,t) \notag\\
  & \ge \frac{c_\psi}t - \alpha^* \, C_\psi \, \sqrt\pi
        - \frac{u^*}2 \, \sqrt\frac{\pi}{C_\ell} \,.
  \label{u.dev.bound}
\end{align}
We conclude that for any weak solution, $u_t$ is strictly positive in
the interior of $\ES(T_{\unique})$, where
\begin{equation}
  \label{T.bar}
  T_{\unique} 
  = \min \biggl\{ T_2, c_\psi \Big/ 
    \Bigl( \alpha^* \, C_\psi \, \sqrt\pi 
    + \frac{u^*}2 \, \sqrt\frac{\pi}{C_\ell} \Bigr) 
    \biggr\} 
\end{equation}
independent of the weak solution $(u,p)$.

Now suppose that $(u_1,p_1)$ and $(u_2,p_2)$ are weak solutions of
\eqref{e.original}.  We claim that, on
$D_{\unique}=\R\times[0,T_{\unique}]$, 
\begin{equation}
  \label{e.claim}
  (p_1u_1-p_2u_2)(u_1-u_2)_+ \geq 0 \,.
\end{equation}
We prove this claim separately on three subdomains.  On
$D_o \cap D_{\unique}$, $p_1 = p_2 = 1$ because $T_{\unique}$ is
selected such that the $x$-projection of this set is included in the
first ring.  Hence, the claim is obvious.  On $D^* \cap D_{\unique}$
and on its symmetric counterpart in the left half-plane,
$p_1 = p_2 = 0$ due to Lemma~\ref{u.psi}; the claim is also obvious.
Finally, on $\ES(T_{\unique})$, we note that
$(p_1u_1-p_2u_2)(u_1-u_2)_+$ can be negative only if
$u_1(x,t)>u_2(x,t)$ and $p_1(x,t)<1$.  By Lemma~\ref{i.prop.7}, we may
assume that $p_1$ and $p_2$ are of the form \eqref{p.redefined}.
Therefore, $p_1(x,t)<1$ implies $u^*\ge u_1(x,t)$.  But then
$u^*\ge u_1(x,t)>u_2(x,t)$. Since $u_2$ is increasing in time on
$\ES(T_{\unique})$, we have $u^*>u_2(x,s)$ if
$(x,s)\in \ES(t)\subset \ES(T_{\unique})$.  So precipitation cannot
start at spatial coordinate $x$ until after time $t$, thus
$p_2(x,t)=0$.  Hence,
$(p_1u_1-p_2u_2)(u_1-u_2)_+=p_1u_1(u_1-u_2)_+\ge0$ at
$(x,t) \in \ES(T_{\unique})$.  This proves \eqref{e.claim}.

We complete the proof with a direct energy estimate.  Proceeding
formally (a first-principles justification can be found in
\cite{Darbenas:2018:PhDThesis}), we note that 
\begin{equation}
\label{u_1-u_2}
  (u_1-u_2)_t = (u_1-u_2)_{xx} - p_1 \, u_1 + p_2 \, u_2 \,,
\end{equation}
multiply with 
$(u_1-u_2)_+$, integrate in space and then integrate by parts, 
\begin{multline}
  \frac12\frac\d{\d t} \int_\R(u_1-u_2)_+^2 \, \d x \\
  = - \int_\R \I_{\{u_1>u_2\}} \,(u_1-u_2)_x^2 \, \d x
    - \int_\R (p_1u_1-p_2u_2)(u_1-u_2)_+ \, \d x \le 0 \,.
  \label{the.trick}
\end{multline}
Integrating in time with $u_1(x,0)-u_2(x,0)=0$, we find that
$u_2\ge u_1$ on $D_{\unique}$.  As the argument is symmetric under
exchange of indices, we also have the reverse inequality, so $u_1=u_2$
on $D_{\unique}$.  An easy argument shows that the precipitation
function is essentially determined by the concentration field (e.g.\
\cite[Lemma~3]{DarbenasHO:2018:LongTA}), hence $p_1 = p_2$ a.e.\ on
$D_{\unique}$.
\end{proof}

\begin{lemma} \label{l.transversality2}
Let $(u,p)$ be a weak solution to \eqref{e.original} with ring domain
$\RD$.  Suppose that there exists $X \leq X^*$ such that
\begin{equation}
  \label{e.transversality2}
  \limsup_{k \searrow 0} \Delta_k^- u(x,t)
  \equiv \limsup_{k \searrow 0} \frac{u(x,\ell(x))-u(x,\ell(x)-k)}k
  > 0
\end{equation} 
for all $x \in D \equiv I(u) \cap (0,X)$.  Then the one-sided
derivatives $u_{x+}(x,\ell(x))$ and $u_{t-}(x,\ell(x))$ exist for all
$x \in D$ with $u_{x+}(x,\ell(x))<0$ and $u_{t-}(x,\ell(x))>0$.
\end{lemma}

\begin{remark}
In contrast to the local statement in Lemma~\ref{Lambda.normal.cond},
the transversality condition \eqref{e.transversality2} here must be
satisfied globally on the domain $I(u) \cap (0,X)$.
\end{remark}

\begin{remark}
At points $(x,\ell(x))$ on the precipitation boundary that do not lie
on the parabola $\Pc$, the one-sided derivatives in
Lemma~\ref{l.transversality2} are regular two-sided derivatives.  For
$u_x$, this follows directly from the concept of weak solution, for
$u_t$, this is a consequence of Lemma~\ref{Lambda.normal.cond}.
\end{remark}

\begin{remark} \label{r.initial} In the proof of
Theorem~\ref{u.uniqueness}, we have already proved that classical
first derivatives exist, with $u_{x}(x,\ell(x))<0$ and
$u_{t}(x,\ell(x))>0$, on the part of the precipitation boundary
contained in $D_{\unique}$.
\end{remark}

\begin{remark}
So long as one of the transversality conditions from
Lemma~\ref{l.transversality2} or Lemma~\ref{Lambda.normal.cond} is
satisfied, thus at least for some initial interval of time, $u$ is
continuously differentiable in time away from the parabola $\Pc$.
Thus, the discontinuity of the precipitation term in the HHMO-model must
be balanced by a discontinuity of $u_{xx}$ across the precipitation
boundary.  This behavior is not obvious from a direct inspection of
the PDE.
\end{remark}

\begin{proof}
Take $x\in(0,X^*)$ such that for all $y \in I(u) \cap (0,x)$, the
one-sided derivatives $u_{x+}(y,\ell(y))$ and $u_{t-}(y,\ell(y))$
exist with $u_{x+}(y,\ell(y))<0$ and $u_{t-}(y,\ell(y))>0$.  (Such an
$x$ exists, see Remark~\ref{r.initial}.) Suppose further that the
transversality condition \eqref{e.transversality2} remains satisfied
at $x$.  We shall show that this implies that $u_{x+}(y,\ell(y))$ and
$u_{t-}(y,\ell(y))$ exist with $u_{x+}(y,\ell(y))<0$ and
$u_{t-}(y,\ell(y))>0$ in a neighborhood of $x$ that is relatively open
in $I(u) \cap (0,X)$.  This implies the lemma as stated.

In the following, set $t=\ell(x)$.  Our main the main task is to show
that $u_{x+}(x,t)<0$, a claim which we prove in three distinct cases
below.  Once this is established, Lemma~\ref{Lambda.normal.cond}
implies that $(x,t)$ is a point of continuity of $(u-\psi)_t$; in
particular, $u_{t-}(x,t)$ is defined and is positive.  When $x$ is the
right boundary point of a ring, this is all we have to show.
Otherwise, we assert that $\ell$ is right-continuous at $x$.  Indeed,
when $(x,t) \in \Pc$, this is trivial.  When $(x,t) \notin \Pc$,
$u_t(x,t)$ is defined and strictly positive, so that $u(x,t+k)>u^*$
for every sufficiently small $k>0$, $(x,t) \notin \Lambda_\jump$, and
$\ell$ is continuous at $(x,t)$.  Right-continuity of $\ell$ at $x$
implies that the one-sided derivatives exist $u_{x+}(y,\ell(y))$ and
$u_{t-}(y,\ell(y))$ exist with their signs preserved in a right
neighborhood of $x$, which completes the argument.

\begin{case}
$u_{x+}(x,t)<0$ if $(x,t) \in \Pc$ and $x$ is not the left boundary
point of a ring.
\end{case}

Take $h>0$ small enough so that $x-h$ is contained in the same ring.
As in the proof of Lemma~\ref{Lambda.normal.cond}, 
\begin{equation}
  u(x,t) = u^* = u(x-h,\ell(x-h)) \,,
\end{equation}
so that
\begin{equation}
  \frac{\ell(x) - \ell(x-h)}h \cdot
  \frac{u(x,t) - u(x,\ell(x-h))}{\ell(x) - \ell(x-h)}
  = - \frac{u(x,\ell(x-h))-u(x-h,\ell(x-h))}h \,.
  \label{e.splitting5}
\end{equation}
Noting that $\ell(x) = x^2/\alpha^2$ and
$\ell(x-h) \leq (x-h)^2/\alpha^2$, so that
$\ell(x) - \ell(x-h) \geq ({2xh} - {h^2})/{\alpha^2}$, we find that
for $h$ sufficiently small,
\begin{equation}
  \frac{\ell(x) - \ell(x-h)}h \geq
  \frac{x}{\alpha^2} > 0 \,.
  \label{e.growth-bound}
\end{equation}
By Lemma~\ref{I.l.prop}\ref{i.prop.3}, $\ell$ is left-continuous and
strictly increasing.  Due to the transversality condition
\eqref{e.transversality2}, this implies that
\begin{equation}
  \limsup_{h \searrow 0}
  \frac{u(x,t) - u(x,\ell(x-h))}{\ell(x) - \ell(x-h)} > 0 \,.
  \label{e.limsup2}
\end{equation}
Last, as $\ell$ is strictly increasing, the open line segment
$\{ (\xi, \ell(x-h)) \colon x-h < \xi < x \}$ lies below the
precipitation boundary for every such $h$.  Since $u_x$ is continuous
on this line segment, the mean value theorem yields
\begin{equation}
  \frac{u(x-h,\ell(x-h)) - u(x, \ell(x-h))}h
  = u_x (\xi(h), \ell(x-h))
\end{equation}
for some $\xi(h) \in (x-h,x)$.  Using left-continuity of $\ell$ and
the fact that $u-\psi$ is continuously differentiable in $x$, we find
\begin{equation}
  \lim_{h \searrow 0} \frac{u(x-h,\ell(x-h)) - u(x, \ell(x-h))}h
  = u_{x+}(x,t) \,.
  \label{e.uxplus}
\end{equation}
Thus, letting $h \searrow 0$ in \eqref{e.splitting5} and referring to
\eqref{e.growth-bound}, \eqref{e.limsup2}, and \eqref{e.uxplus} for each of
the terms, we conclude that $u_{x+}(x,t)<0$.

\begin{case}
$u_{x+}(x,t)<0$ if $x=X_{2i}$, i.e., $x$ is the starting location of a
ring.
\end{case}

In this case, $(x,t) \in \Pc$ by Lemma~\ref{I.l.prop}\ref{i.prop.3a}
and the location of the singularity of the heat kernel in the Duhamel
integral is bounded away from the effective domain of integration, so
that we can differentiate the Duhamel formula directly to find
\begin{equation}
  u_{t-}(x,t)
  = \psi_{t-}(x,t)
    - \int_0^t\int_\R \Phi_t(x-y,t-s) \, p(y,s) \,
        u(y,s) \, \d y \, \d s \,.
\end{equation}
This shows that $u_{t-}$ exists and is continuous on the ray
$[x,\infty)\times\{t\}$.  To proceed, we recall that $\ell$ is
increasing and left-continuous, so that on any box with upper left
corner $(x,t)$, $p=0$ so that $u$ solves the heat equation
$u_t=u_{xx}$.  Then, by the Taylor formula with integral remainder,
\begin{equation}
  u(x+h,t) = u(x,t) + u_{x+}(x,t) \, h
    + \int_{x}^{x+h} (x+h-\xi) \, u_t(\xi,t) \, \d \xi \,.
  \label{e.taylor}
\end{equation}
Since $u^* = u(x,t) \geq u(x+h,t)$ and the integral in
\eqref{e.taylor} is strictly positive for $h$ small enough due to
continuity of $u_{t-}$ and the transversality condition
\eqref{e.transversality2}, we conclude that $u_{x+}(x,t)<0$.

\begin{case}
$u_{x+}(x,t)<0$ if $(x,t) \notin \Pc$.
\end{case}

\begin{figure}
\centering
\includegraphics{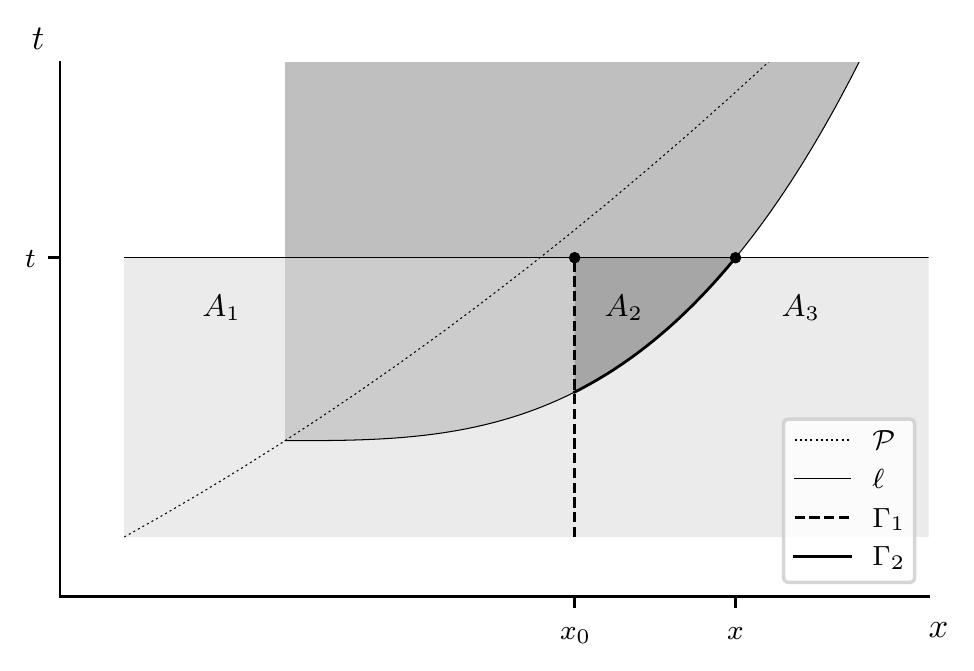}
\caption{Sketch of the geometry of the construction used in the proof
of Lemma~\ref{l.transversality2}, Case~3.}
\label{f.transversality-sketch}
\end{figure}

This case cannot be solved by a local argument, as we have no lower
bound on the growth of $\ell$ as in \eqref{e.growth-bound}.  We take
$x_0<x$ large enough such that $x_0$ is the same ring as $x$ and
$(x_0,t)$ lies below the parabola $\Pc$.  We split the space-time
domain into three subregions, see
Figure~\ref{f.transversality-sketch}:
\begin{subequations}
\begin{gather}
  A_1 = (-\infty,x_0) \times (0,t) \,, \\
  A_2 = \{ (y,s) \colon x_0<y<x, \ell(y) < s \leq t \} \,, \\
  A_3 = \{ (y,s) \colon x_0<y, 0 < s < \min \{ \ell(y), t \} \} \,.
\end{gather}
\end{subequations}

We now proceed in three steps.  In the first step, we show that $u_t$
is bounded on $A_2$.  By Lemma~\ref{u-psi.non-incr}, $u_t$ is bounded
above, so it suffices to find a lower bound.  We first note that on
$\Gamma_1$, the right boundary of $A_1$, $u_t$ is continuous up to the
boundary points, hence is bounded.  On $\Gamma_2$, the joint boundary
of $A_2$ and $A_3$ we have $u_t>0$ by assumption except perhaps at the
end point $(x,t)$ where we do not know yet whether $u_t$ is defined.
(Recall that the continuation argument implies that $\ell$ is
continuous at every $y<x$, so that every point on $\Gamma_2$ is of the
form $(y, \ell(y))$, thus covered by the transversality condition
\eqref{e.transversality2}.)  Noting that $v = u_t$ satisfies the
equation $v_t = v_{xx} - v$ on $A_2$, we invoke the parabolic maximum
principle to conclude that $v$ is bounded on $A_2$.  (A similar
argument can be made on $A_3$ where $v$ satisfies the heat equation,
but this will not be necessary in the following as $p=0$ on this
region.)

In the second step, we show that
\begin{equation}
  0 < \limsup_{k \searrow 0} \Delta_k^{-} u(x,t)
  \leq \psi_{t-} (x,t) - \Fc_1 (x,t) - u^* \, \Fc_2 (x,t) \,.
  \label{e.udeltakm}
\end{equation}
This inequality implies, in particular, that $\Fc_2 (x,t)$ is finite.
The left inequality is simply restating the temporal transversality
condition \eqref{e.transversality2}.  To prove the right inequality in
\eqref{e.udeltakm}, we take an arbitrary $r \in (x_0,x)$.  Recalling
the Duhamel formula \eqref{e.duhamel}, splitting the spatial domain of
integration, changing the time variable in the integral corresponding
to the right spatial subdomain, and noting that, by
Lemma~\ref{i.prop.7}, $p$ is non-decreasing in time, we find, for
$k \geq 0$, that
\begin{align}
  u(x,t)
  & \leq \psi(x,t) - \int_0^t \int_r^\infty \Phi(x-y, s) \,
           p(y,t-s-k) \, u(y,t-s) \, \d y \, \d s
         \notag \\
  & \quad - \int_0^t \int_{-\infty}^r \Phi(x-y, t-s) \,
           p(y,s) \, u(y,s) \, \d y \, \d s \,.
\end{align}
(We imply that $p(x,t)=0$ for $t < 0$.)  Similarly,
\begin{align}
  u(x,t-k)
  & = \psi(x,t-k) - \int_0^t \int_r^\infty \Phi(x-y, s) \,
           p(y,t-s-k) \, u(y,t-s-k) \, \d y \, \d s
         \notag \\
  & \quad - \int_0^t \int_{-\infty}^r \Phi(x-y, t-s-k) \,
           p(y,s) \, u(y,s) \, \d y \, \d s \,.
\end{align}
(As before, we understand that $\Phi(x,t)=0$ for $t<0$.)  Then
\begin{align}
  \Delta_k^- u(x,t)
  & \leq \Delta_k^- \psi(x,t)
    -  \int_0^t \int_r^\infty \Phi(x-y, s) \,
           p(y,t-s-k) \, \Delta_k^- u(y,t-s) \, \d y \, \d s
         \notag \\
  & \quad - \int_0^t \int_{-\infty}^r \Delta_k^- \Phi(x-y, t-s) \,
           p(y,s) \, u(y,s) \, \d y \, \d s \,.
\end{align}
We now take the limit $k \searrow 0$ and apply the dominated
convergence theorem to each of the integrals.  For the first integral,
existence of a dominating function follows from boundedness of $u_t$
on $A_2$ and the fact that $p=0$ on $A_3$. For the second integral, we
note that the domain of integration is bounded away from the
singularity of the heat kernel and that $p$ is compactly supported.
Thus,
\begin{align}
  \limsup_{k\searrow 0}
  \Delta_k^- u(x,t)
  & \leq \psi_{t-}(x,t)
    -  \int_0^t \int_r^\infty \Phi(x-y, s) \,
           p(y,t-s) \, u_t(y,t-s) \, \d y \, \d s
         \notag \\
  & \quad - \int_0^t \int_{-\infty}^r \Phi_t(x-y, t-s) \,
           p(y,s) \, u(y,s) \, \d y \, \d s
         \notag \\
 & = \psi_{t-}(x,t) - \Fc_1(x,t) - u^* \int_{-\infty}^{r} \I_{I}(y) \,
         \Phi(x-y,t-\ell(y)) \, \d y \,,
\end{align}
where the last equality is due to integration by parts as in Step~3 in
the proof of Theorem~\ref{cont.theo}.  Letting $r \nearrow x$, we
obtain \eqref{e.udeltakm} by monotone convergence.

Finally, as the point $(x+h,t)$ lies below $\Lambda_\normal$ and below
the parabola $\Pc$, Theorem~\ref{cont.theo} applies, i.e.,
\begin{equation}
  u_t(x+h,t) = \psi_t(x+h,t) - \Fc_1(x+h,t) - u^* \, \Fc_2(x+h,t) \,.
  \label{e.ut_xph}
\end{equation}
Noting that $\psi_t$ is right-continuous in $x$, $\Fc_1$ is continuous
at $(x,t)$ (the convolution restricted to $A_2 \cup A_3$ is continuous
as a convolution of an $L^1$ with an $L^\infty$ function as $p \, u_t$
is bounded; the convolution restricted to $A_1$ is continuous as the
singularity of the kernel is located away from the support of the
integrand), and $\Fc_2(x+h,t)$ is monotonically increasing and bounded
by $\Fc_2(x,t)$, so that
\begin{equation}
  \liminf_{h \searrow 0} u_t(x+h,t)
  \geq \psi_{t-}(x,t) - \Fc_1(x,t) - u^* \, \Fc_2 (x,t) \,.
\end{equation}
Using \eqref{e.udeltakm}, we find that the right hand side is strictly
positive.  This shows that the integral in \eqref{e.taylor} is
strictly positive for $h$ small enough, so that we can finish the
proof as in Case~2.  This concludes the proof of
Lemma~\ref{l.transversality2}.
\end{proof}

\begin{remark}
Under the conditions of Lemma~\ref{l.transversality2}, it is easy to
show that $\ell$ is continuously differentiable on
$I(u) \setminus \{0\}$ with
\begin{equation}
  \ell'(x) = -\frac{u_{x+}(x,\ell(x))}{u_{t-}(x,\ell(x))} \,.
  \label{e.ellprime}
\end{equation}
Indeed, the continuation argument in the proof of
Lemma~\ref{l.transversality2} yields continuity of $\ell$ on $I(u)$.
Further, as $u=u^*$ on the precipitation boundary,
\begin{equation}
  u(x, \ell(x)) = u^* = u(x-h, \ell(x-h))
\end{equation}
and therefore
\begin{equation}
  \frac{\ell(x)-\ell(x-h)}h \cdot
  \frac{u(x, \ell(x)) - u(x,\ell(x-h))}{\ell(x)-\ell(x-h)}
  = - \frac{u(x,\ell(x-h)) - u(x-h, \ell(x-h))}h
\end{equation}
Since $u_{x+}$ is continuous, the right hand fraction converges to
$u_{x+}(x,t)$ as $h \searrow 0$ by the mean value theorem.  By
continuity of $\ell$, the second fraction on the left converges to
$u_{t-}$, which is non-zero by Lemma~\ref{l.transversality2}.  This
proves that the $\ell_{x-}$ satisfies \eqref{e.ellprime}; the argument
for $\ell_{x+}$ is similar.
\end{remark}


\begin{theorem}[Conditional uniqueness] \label{t.cond-uniqueness}
Suppose $(u_1,p_1)$ and $(u_2,p_2)$ are two weak solutions to the
HHMO-model \eqref{e.original} with ring domains $\RD_1$ and $\RD_2$,
respectively.  Assume that $u_2$ satisfies the temporal transversality
condition \eqref{e.transversality2} with
$X \leq \min \{X_1^*, X_2^* \}$, where $X_1^*$ and $X_2^*$ are the
respective spatial extents of precipitation on the two ring domains.
Then $u_1 = u_2$ on $\R \times (0,(X/\alpha)^2)$ and $p_1=p_2$ a.e.\
on this domain.
\end{theorem}

\begin{proof}
Suppose the contrary.  Then there exists
$t^* \in [T_\unique, (X/\alpha)^2)$ such that $u_1=u_2$ on
$\R \times [0,t^*]$ and $t^*$ is maximal with this property.  By
uniqueness of solutions for linear parabolic equations, the
concentrations $u_1$ and $u_2$ can only differ at time $t$ if the
precipitation functions $p_1$ and $p_2$ differ on a subset of
$\R \times [0,t]$ of positive space-time measure.  Further, by
Lemma~\ref{i.prop.7}, $p_1$ and $p_2$ are essentially determined by
the respective precipitation fronts $\ell_1$ and $\ell_2$, and we
assume their canonical representation given by \eqref{p.redefined}
henceforth.  Thus, there must be $x^* < X$ such that
$\ell_1(x)=\ell_2(x)$ for $x \leq x^*$ and $\ell_1(x) \neq \ell_2(x)$
for some $x$ in every right neighborhood of $x^*$.  (For ease of
notation, we take $\ell_i(x) = \infty$ if $x \notin I(u_i)$.)

We claim that $\ell_1$ and $\ell_2$ are ``entangled'' in the sense
that in every right neighborhood of $x^*$ there exist points where
$\ell_1<\ell_2$ as well as points where $\ell_2<\ell_1$.  If not,
there were a right neighborhood $[x^*, x^*+\eps)$ on which the
precipitation fronts were ordered, $\ell_1 \leq \ell_2$, say, with
strict inequality somewhere in every right neighborhood of $x^*$; by
maximality of $t^*$ and monotonicity of $\ell_1$,
$\ell_1(x^*+h) \searrow t^*$ as $h \searrow 0$. But then $p_1 \geq p_2$
so that $u_1 \leq u_2$ on $\R \times [t^*,\ell_1(x^*+\eps))$ by the
parabolic comparison principle and therefore $\ell_1 \geq \ell_2$ on
$[x^*,x^*+\eps)$, a contradiction.

Moreover, the energy estimate in the last part of the proof of
Theorem~\ref{u.uniqueness}, following \eqref{e.claim}, shows that
$u_1$ can only exceed $u_2$ somewhere for every $t>t^*$ if
\begin{equation}
  (p_1u_1-p_2u_2)(u_1-u_2)_+ < 0
  \label{e.nonuniquecondition}
\end{equation}
somewhere in every neighborhood of $(x^*,t^*)$.  This can only happen
at points where $p_1=0$, $p_2=1$, and $u_2 < u_1 \leq u^*$.  Thus,
$u_2$ must be decreasing somewhere in every neighborhood of
$(x^*,t^*)$.  But, by transversality and
Lemma~\ref{l.transversality2}, $(u_2)_{t-}(x^*,t^*) > 0$ so that, by
continuity of the time derivative on $D_u$, $u_2$ must be strictly
increasing in some neighborhood of $(x^*,t^*)$ below the parabola
$\Pc$.  If $(x^*,t^*) \notin \Pc$, this is in immediate contradiction.
If $(x^*,t^*) \in \Pc$, this means that the locations where
\eqref{e.nonuniquecondition} occurs must lie in $D_o$, thus within a
gap of $u_1$.  Thus, $u_1$ must have an infinite number of gaps in
every right neighborhood of $x^*$, which is not permitted on its ring
domain.
\end{proof}

\begin{remark}
The proof gives clear constraints on how solutions might be continued
in non-unique ways.  Within a ring domain, so at least for the initial
part of the evolution, non-uniqueness requires ``entanglement'' of the
precipitation fronts of the two different solutions.  Past the point of
breakdown of the ring domain, which can be shown to occur in similar
models and which is conjectured to occur for the HHMO-model as well
based on numerical studies, the possibilities in which non-uniqueness
might occur are less constrained \cite{DarbenasO:2021:BreakdownLP}.
It could come about, e.g., via different ways of accumulating an
infinite number of precipitation rings in right neighborhoods of a
critical point $x^*$.  Such scenarios remain a possible even for
generalized solutions to the related scalar model problem discussed in
\cite{DarbenasO:2021:BreakdownLP}, and it is open whether there is a
natural selection principle for such generalized solutions that will
lead to unique continuation.
\end{remark}

\section*{Acknowledgments}

We thank Danielle Hilhorst for insightful discussions.  This work was
funded through German Research Foundation grant OL 155/5-1.
Additional funding was received via the Collaborative Research Center
TRR 181 ``Energy Transfers in Atmosphere and Ocean'', also supported
by the German Research Foundation, also funded by the DFG under
project number 274762653.

\bibliographystyle{acm}
\bibliography{liesegang}

\end{document}